\documentclass[a4paper,12pt]{amsart}
\usepackage{amsfonts,amsthm,amssymb,amsmath,amscd}
\usepackage[colorlinks=true,linkcolor=blue,citecolor=blue]{hyperref}
\usepackage{amsmath}
\usepackage{tikz-cd}

\usepackage{tabularx}
\usepackage{tcolorbox}
\usepackage[english]{babel}
\usepackage{amssymb,amsmath}
\usepackage{xcolor}

\newtheorem{theorem}{Theorem}[section]
\newtheorem{proposition}[theorem]{Proposition}
\newtheorem{corollary}[theorem]{Corollary}
\newtheorem{lemma}[theorem]{Lemma}
\newtheorem{definition}[theorem]{Definition}

\newtheorem*{mainresult}{Main result}
\newtheorem*{Bohrstheorem}{Bohr's theorem}

\newtheorem{remark}[theorem]{Remark}

\begin{document}

\title{On Bohr's theorem for general Dirichlet series}

\author[Schoolmann]{Ingo Schoolmann}
\address[]{Ingo Schoolmann\newline  Institut f\"{u}r Mathematik,\newline Carl von Ossietzky Universit\"at,\newline
26111 Oldenburg, Germany.
}
\email{ingo.schoolmann@uni-oldenburg.de}

\maketitle

\begin{abstract}
\noindent
We present quantitative versions of Bohr's theorem on general Dirichlet series $D=\sum a_{n} e^{-\lambda_{n}s}$ assuming different assumptions on the frequency $\lambda=(\lambda_{n})$, including the conditions introduced by Bohr and Landau. Therefore, using the summation method by typical (first) means invented by M. Riesz, without any condition on $\lambda$, we give upper bounds for the norm of the partial sum operator $S_{N}(D):=\sum_{n=1}^{N} a_{n}(D)e^{-\lambda_{n}s}$ of length $N$ on the space $\mathcal{D}_{\infty}^{ext}(\lambda)$ of all somewhere convergent $\lambda$-Dirichlet series, which allow a holomorphic and bounded extension to the open right half plane $[Re>0]$. As a consequence for some classes of $\lambda$'s we obtain a Montel theorem in $\mathcal{D}_{\infty}(\lambda)$; the space of all $D \in \mathcal{D}_{\infty}^{ext}(\lambda)$ which converge on $[Re>0]$. Moreover, following the ideas of Neder we give a construction of frequencies $\lambda$ for which $\mathcal{D}_{\infty}(\lambda)$ fails to be complete.

\end{abstract}
\maketitle                   






\section{Introduction}
A general Dirichlet series is a formal sum $\sum a_{n} e^{-\lambda_{n}s}$, where $(a_{n})$ are complex coefficients (called Dirichlet coefficients), $s$ a complex variable and $\lambda:=(\lambda_{n})$ a strictly increasing non negative real sequence tending to $+\infty$  (called frequency). To see  first examples we choose $\lambda=(\log n )$ and obtain ordinary Dirichlet series $\sum a_{n} n^{-s}$, whereas the choice $\lambda=(n)=(0,1,2,\ldots)$ leads to formal power series $\sum a_{n} z^{n}$ regarding the substitution $z=e^{-s}$.\\

 Within the last two decades the theory of  ordinary Dirichlet series had a sort of renaissance which in particular led to the solution of some long-standing problems (see \cite{Defant} and \cite{QQ} for more information). A fundamental object in these investigations
 is given by the Banach space $\mathcal{H}_{\infty}$  of all ordinary Dirichlet series $D:=\sum a_{n} n^{-s}$, which converge and define bounded functions on $[Re>0]$.\\

One of the main tools in this theory is the fact that every ordinary Dirichlet series $D \in \mathcal{H}_{\infty}$ converges uniformly on $[Re>\varepsilon]$ for all $\varepsilon>0$, which is a consequence of what is called Bohr's theorem and was proven by Bohr in \cite{Bohr}.
\begin{Bohrstheorem}[qualitative version] \label{Bohrstheoordinaryquali} Let $D=\sum a_{n} n^{-s}$ be a somewhere convergent ordinary Dirichlet series having a holomorphic and bounded extension $f$ to $[Re>0]$. Then $D$ converges uniformly on $[Re>\varepsilon]$ for all $\varepsilon>0$.
\end{Bohrstheorem}
Several years ago in \cite{BalasubramanianCaladoQueffelec} this 'ordinary' result was improved by a quantitative version.
\begin{Bohrstheorem}[quantitative version] There is a constant $C>0$ such that for all somewhere convergent $D=\sum a_{n} n^{-s}$ allowing a holomorphic and bounded extension $f$ to $[Re>0]$ and for all $N \in \mathbf{N}$ with $N\ge 2$
\begin{equation} \label{BTordinary}
\sup_{[Re>0]} \left| \sum_{n=1}^{N} a_{n} n^{-s}\right| \le C \log(N) \sup_{[Re>0]} |f(s)|.
\end{equation}
\end{Bohrstheorem}
\noindent  An important consequence is that Bohr's theorem implies that $\mathcal{H}_{\infty}$ is a Banach space (see \cite[\S1.4]{Defant}). \\

The natural domain of Bohr's theorem for general Dirichlet series is the space $\mathcal{D}_{\infty}^{ext}(\lambda)$ of all somewhere convergent $\lambda$-Dirichlet series $D=\sum a_{n}e^{-\lambda_{n}s}$ allowing a holomorphic and bounded extension $f$ to $[Re>0]$. Additionally we define the subspace $\mathcal{D}_{\infty}(\lambda)$ of all $D \in \mathcal{D}_{\infty}^{ext}(\lambda)$ which converge on $[Re>0]$. Notice that with this notation we have $\mathcal{D}_{\infty}((\log n))=\mathcal{H}_{\infty}$. The inclusion $\mathcal{D}_{\infty}(\lambda) \subset \mathcal{D}^{ext}_{\infty}(\lambda)$ in general is strict (see e.g. the frequencies constructed  in \cite[\S1]{Neder}). A natural norm on $\mathcal{D}_{\infty}^{ext}(\lambda)$ (and on $\mathcal{D}_{\infty}(\lambda)$) is given by $\|D\|_{\infty}:=\sup_{[Re>0]} |f(s)|$, where $f$ is the (unique) extension of $D$. Note that a priori, $\|\cdot\|_{\infty}$ is only a semi norm, that is it could be possible for a particular Dirichlet series $D=\sum a_{n}e^{-\lambda_{n}s}$ with some $a_{n}\ne 0$ to have a bounded holomorphic extension $f$ with $\|D\|_{\infty}=0$, or equivalently, it is not clear whether
 $\mathcal{D}^{ext}_{\infty}(\lambda)$ can be considered as a subspace of $H_{\infty}[Re>0]$, the Banach space of all holomorphic and bounded functions on $[Re>0]$. Here it is important to distinguish Dirichlet series from their limit function, and  to prove that $\|\cdot\|_{\infty}$  in fact is a norm on $\mathcal{D}^{ext}_{\infty}(\lambda)$  requires to check
that  all Dirichlet coefficients of $D$ vanish provided $\|D\|_\infty =0$ (see Corollary \ref{almost periodic}).\\

We say that a frequency $\lambda$ satisfies Bohr's theorem (or Bohr's theorem holds for $\lambda$) if every $D\in \mathcal{D}^{ext}_{\infty}(\lambda)$ converges uniformly on $[Re>\varepsilon]$ for all $\varepsilon>0$. It was a prominent question in the beginning of the 20th century for which $\lambda$'s Bohr's theorem hold.\\

Actually Bohr proves his theorem not only for the case $\lambda=(\log n)$ but for the class of $\lambda$'s satisfying the following condition (we call it Bohr's condition $(BC)$):
\begin{equation} \label{BC}
\exists ~l = l (\lambda) >0 ~ \forall ~\delta >0 ~\exists ~C>0~\forall~ n \in \mathbf{N}: ~~\lambda_{n+1}-\lambda_{n}\ge Ce^{-(l+\delta)\lambda_{n}};
\end{equation}
roughly speaking this condition prevents the $\lambda_n$'s from getting too close too fast. Then in \cite{Bohr} Bohr shows that if $\lambda$ satisfies $(BC)$, then Bohr's theorem hold for $\lambda$. Note that $\lambda=(\log n)$ satisfies $(BC)$ with $l=1$.\\

In \cite{Landau} Landau gives another sufficient condition (we call it Landau's condition $(LC)$), which is weaker than $(BC)$ and extends the class of frequencies for which Bohr's theorem hold:
\begin{equation} \label{LC}
\forall ~\delta>0 ~\exists ~C>0~ \forall ~n \in \mathbf{N} \colon ~~ \lambda_{n+1}-\lambda_{n}\ge C e^{-e^{\delta\lambda_{n}}}.
\end{equation} We like to mention that in \cite[\S 1]{Neder} Neder goes a step further and considered $\lambda$'s satisfying
\begin{equation*}
\exists ~x>0 ~\exists ~C>0~ \forall ~n \in \mathbf{N} \colon ~~ \lambda_{n+1}-\lambda_{n}\ge C e^{-e^{x\lambda_{n}}}.
\end{equation*}
Then Neder proved that this condition is not sufficient for satisfying Bohr's theorem by constructing, giving some $x>0$, a Dirichlet series $D$ (belonging to some frequency $\lambda$) for which $\sigma_{c}(D)=\sigma_{a}(D)=x$ and $\sigma_{b}^{ext}(D)\le 0$ hold. In particular this shows that the inclusion $\mathcal{D}_{\infty}(\lambda) \subset \mathcal{D}^{ext}_{\infty}(\lambda)$ is strict for these $\lambda$'s.\\

Like Bohr, Landau under his condition $(LC)$ only proves the qualitative version of Bohr's theorem. Of course, establishing quantitative versions means to control the norm of the partial sum operator
\begin{equation}  \label{werder}
S_{N} \colon \mathcal{D}_{\infty}^{ext}(\lambda) \to \mathcal{D}_{\infty}(\lambda), ~~ D \mapsto \sum_{n=1}^{N}a_{n}(D) e^{-\lambda_{n}s}, ~ N \in \mathbf{N},
\end{equation}
since by definition
\begin{equation*}
\sup_{[Re>0]} \left| \sum_{n=1}^{N} a_{n} e^{-\lambda_{n}s} \right| \le \|S_{N}\| \|D\|_{\infty}.
\end{equation*}
Then using the summation method of typical means of order $k>0$ invented by M. Riesz (Proposition \ref{approx}), our main result gives an estimate of $\|S_{N}\|$ without assuming any condition on $\lambda$ (Theorem \ref{keylemma}).
\begin{mainresult} For all $0<k\le 1$ and $N\in \mathbf{N}$ we have
\begin{equation} \label{keyestimate}
\|S_{N}\|\le C \frac{\Gamma(k+1)}{k} \left(\frac{\lambda_{N+1}}{\lambda_{N+1}-\lambda_{N}} \right)^{k},
\end{equation}
where $\Gamma$ is the Gamma function and $C>0$ a universal constant.
\end{mainresult}
As a consequence assuming Bohr's condition $(\ref{BC})$ on $\lambda$ the choice $k_{N}:=\frac{1}{\lambda_{N}}$, $N\ge 2$ (since $\lambda_{1}=0$ is possible), leads to
\begin{equation*}
\|S_{N}\|\le C_{1}(\lambda) \lambda_{N},
\end{equation*}
which reproves $(\ref{BTordinary})$ for $\lambda=(\log n)$. Under Landau's condition $(\ref{LC})$ using $(\ref{keyestimate})$ with $k_{N}:=e^{-\delta \lambda_{n}}$, $\delta>0$, we obtain
\begin{equation} \label{ruecken}
\|S_{N}\|\le C_{2}(\lambda, \delta) e^{\delta\lambda_{N}};
\end{equation}
the quantitative version of Bohr's theorem under $(LC)$. As a consequence of $(\ref{ruecken})$ we extend Bayart's Montel theorem from the ordinary case (see \cite[\S 4.3.3, Lemma 18]{Bayart}) to $\mathcal{D}_{\infty}(\lambda)$ in the case of $\lambda$'s satisfying $(LC)$ (Theorem \ref{montel}).\\

Another application of the summation method of typical means gives an alternative proof of the fact that  $\mathbf{Q}$-linearly independent $\lambda$'s (that is $\sum q_{n} \lambda_{n}=0$ implies $q=0$ for all finite rational sequences $q=(q_{n})$) satisfy Bohr's theorem, which was proven by Bohr in \cite{Bohr3}. More precisely we show that in this case the space $\mathcal{D}_{\infty}^{ext}(\lambda)$ equals $\ell_{1}$ (as Banach spaces) via  $\sum a_{n} e^{-\lambda_{n}s} \mapsto (a_{n})$ (Theorem \ref{independent}).\\

Moreover, we would like to consider $\mathcal{D}_{\infty}(\lambda)$ as a Banach space. Unfortunately it may fail to be complete. Based on ideas of Neder we give a construction of $\lambda$'s for which $\mathcal{D}_{\infty}(\lambda)$ is not complete (Theorem \ref{incomplete}). But there are sufficient conditions on $\lambda$, including $(BC)$ and $\mathbf{Q}$-linearly independence, we present in Theorem \ref{completeness}.\\

Before we start let us mention that recently in  \cite{Maestre} given a frequency $\eta$ the authors introduced the space $\mathcal{H}_{\infty}(\eta)$ of all series of the form $\sum b_{n} \eta_{n}^{-s}$, which converge and define a bounded function on $[Re>0]$. Then  defining $\lambda:=(\log(\eta_{n}))$ we have $\mathcal{H}_{\infty}(\eta)=\mathcal{D}_{\infty}(\lambda)$ and so both approaches are equivalent in this sense. 
All results on $\mathcal{H}_{\infty}(\eta)$  in \cite{Maestre} are based on  the assumption that $\lambda$ satisfies the condition $(BC)$.
In contrast to this article, we here try to avoid assumptions on $\lambda$ as much as possible.\\

This text is inspired by the work of (in alphabetical order) Besicovitch, Bohr, Hardy,  Landau, Neder, Perron, M. Riesz. In Section \ref{means} we prove our main result and in Section \ref{Bohrstheosection} we apply it to obtain quantitative variants of Bohr's theorem under different assumptions on $\lambda$, including $(BC)$ and $(LC)$. We finish by Section \ref{completenesschapter}, where we face completeness of $\mathcal{D}_{\infty}(\lambda)$. We start  recalling some basics  on Dirichlet series.
\section{General Dirichlet series} \label{gDs}
As already mentioned in the introduction a strictly increasing non negative real sequence $\lambda:=(\lambda_{n})$ tending to $+\infty$  we call a frequency. Then general Dirichlet series  $D=\sum a_{n} e^{-\lambda_{n}s}$ belonging to some $\lambda$ we call $\lambda$-Dirichlet series and we define $\mathcal{D}(\lambda)$ to be the space of all (formal) $\lambda$-Dirichlet series. Moreover, the (complex) coefficient $a_{n}$ is called the $n$th Dirichlet coefficients of $D$. Finite sums $\sum_{n=1}^{N} a_{n}e^{-\lambda_{n}s}$ are called Dirichlet polynomials.\\

Recall that the natural domains of convergence of Dirichlet series are half spaces (see \cite[Theorem 1, p. 3]{HardyRiesz}). The following 'abscissas' rule the convergence theory of general Dirichlet series.
\begin{align*}
&
\sigma_{c}(D)=\inf\left \{ \sigma \in \mathbf{R} \mid D \text{ converges on } [Re>\sigma] \right\},
\\&
\sigma_{a}(D)=\inf\left \{ \sigma\in \mathbf{R} \mid D \text{ converges absolutely on } [Re>\sigma] \right\},
\\&
\sigma_{u}(D)=\inf\left \{ \sigma \in \mathbf{R} \mid D \text{ converges uniformly on } [Re>\sigma] \right\},
\end{align*}
and
\begin{equation*}
\sigma_{b}(D)=\inf\left \{ \sigma \in \mathbf{R} \mid  D \text{ converges and defines a bounded function on } [Re>\sigma] \right\}.
\end{equation*}
Additionally we define, provided $\sigma_{c}(D)<\infty$,
\begin{align*}
\sigma_{b}^{ext}(D)=\inf\{ \sigma \in \mathbf{R} \mid  &\text{ the limit function of $D$ allows a holomorphic and bounded}\\ &\text{ extension to } [Re>\sigma] \}.
\end{align*}

By definition $\sigma_{c}(D)\le \sigma_{b}(D) \le\sigma_{u}(D)\le \sigma_{a}(D)$ and $\sigma_{b}^{ext}(D)\le \sigma_{b}(D)$. In  general all these abscissas differ. For instance an example of Bohr shows that  $\sigma_{c}(D)=\sigma_{b}^{ext}(D)=\sigma_{b}(D)=-\infty$ and $\sigma_{u}(D)=+\infty$ is possible (see \cite{Bohr4}). Moreover, general Dirichlet series define holomorphic functions on $[Re>\sigma_{c}(D)]$, which relies on the fact that they converge uniformly on all compact subsets of $[Re>\sigma_{c}(D)]$ (see \cite[Theorem 2, p. 3]{HardyRiesz}). Let us recall the spaces of Dirichlet series already defined in the introduction.
\begin{definition} \label{Dinfty} Let $\lambda$ be a frequency. We define the space $\mathcal{D}_{\infty}(\lambda)$ as the space of all $\lambda$-Dirichlet series $D$ which converge on $[Re>0]$ and define a  bounded function there.
\end{definition}
\begin{definition}
We define $\mathcal{D}_{\infty}^{ext}(\lambda)$ to be the  space of all somewhere convergent Dirichlet series $D\in \mathcal{D}(\lambda)$, which allow a holomorphic and bounded extension $f$ to $[Re>0]$.
\end{definition}
We endow $\mathcal{D}^{ext}_{\infty}(\lambda)$ with the semi norm $\|D\|_{\infty}:=\sup_{[Re>0]} |f(s)|$, where $f$ is the (unique) extension of $D$. Corollary \ref{almost periodic} proves that $\|\cdot\|_{\infty}$ in fact is a norm. Bohr's theorem (say in the ordinary case) motivates to give the following definition.
\begin{definition} \label{defbohrtheorem}
We say that a frequency $\lambda$ satisfies Bohr's theorem (or Bohr's theorem hold for $\lambda$), whenever every $D \in \mathcal{D}^{ext}_\infty(\lambda)$ converges uniformly on $[\text{Re} > \varepsilon]$ for  all
$\varepsilon >0$.
\end{definition}
Observe that $\lambda$ satisfies Bohr's theorem if and only if $\sigma_{b}^{ext}=\sigma_{u}$ for all somewhere convergent $\lambda$-Dirichlet series.\\

Let us consider again $\lambda=(n)=(0,1,2,\ldots)$ to see an easy example. Then $(BC)$ holds and via the substitution $z=e^{-s}$, which maps the open right half plane $[Re>0]$ to the open punctured  unit disc $\mathbf{D}\setminus \{0\}$, the space $\mathcal{D}_{\infty}((n))$ coincides with $H_{\infty}(\mathbf{D})$; the space of all holomorphic and bounded functions on $\mathbf{D}$. In this case Bohr's theorem states the fact, that if a power series $P(z)=\sum c_{n} z^{n}$ converging on some neighbourhood of the origin allows an extension $g \in  H_{\infty}(\mathbf{D})$, then $P$ actually converges in $\mathbf{D}$ and coincides with $g$ with  uniform convergence on each closed disk contained in $\mathbf{D}$.
\subsection{A Bohr-Cahen formula}
There are useful Bohr-Cahen formulas for the abscissas $\sigma_{c}$ and $\sigma_{a}$, that are, given $D=\sum a_{n}e^{-\lambda_{n}s}$,
\begin{equation*}
\sigma_{c}(D)\le \limsup_{N} \frac{\log\left( \left| \sum_{n=1}^{N} a_{n}\right| \right) }{\lambda_{N}} ~\text{ and  } ~\sigma_{a}(D)\le \limsup_{N} \frac{\log\left(  \sum_{n=1}^{N} |a_{n}| \right) }{\lambda_{N}},
\end{equation*}
 where in each case equality holds if the left hand side is non negative.
See \cite[Theorem 7 and 8, p. 8]{HardyRiesz} for a proof. The formula for $\sigma_{u}$ (and its proof) extends from the ordinary case in \cite[\S 1.1, Proposition 1.6]{Defant} canonically to arbitrary $\lambda$'s:
\begin{equation} \label{sigmaU}
\sigma_{u}(D)\le \limsup_{N} \frac{\log\left(\sup_{t \in \mathbf{R}} \left|\sum^{N}_{n=1} a_{n}e^{-\lambda_{n}it}\right|\right)}{\lambda_{N}},
\end{equation}
where again the equality holds if the left hand side is non negative. In this section we derive (\ref{sigmaU}) from the following Proposition concerning uniform convergence of sequences of Dirichlet series and take advantage of both in Section \ref{Bohrstheosection}. Then the particular case of a sequence of partial sums will reprove $(\ref{sigmaU})$.\\

 Therefore given a sequence of (formal) $\lambda$-Dirichlet series $D_{j}=\sum a^{j}_{n} e^{-\lambda_{n}s}$ we define
\begin{equation*}
\Delta=\Delta((D_{j})):=\limsup_{(N,j)\in \mathbf{N}^{2}} ~\frac{\log\left(\sup_{t \in \mathbf{R}}\left|\sum_{n=1}^{N}a_{n}^{j} e^{-\lambda_{n}it}\right|\right)}{\lambda_{N}},
\end{equation*}
where we endow $\mathbf{N}^{2}$ with the product order, that is $(a,b)\le (c,d)$ if and only if $a\le c$ and $b\le d$.
\begin{proposition} \label{BohrCahen}
Let $\lambda$ be frequency and $D_{j}=\sum a_{n}^{j} e^{-\lambda_{n}s}$ a sequence of $\lambda$-Dirichlet series, such that the limits $a_{n}:= \lim_{j\to \infty}a_{n}^{j}$ exist for all $n$.
Then $(D_{j})$ converges uniformly on $[Re >\Delta+\varepsilon]$ to $D:=\sum a_{n} e^{-\lambda_{n}s}$ for all $\varepsilon >0$.
\end{proposition}

Dealing with uniform convergence on half spaces it is enough to check on vertical lines, since for any finite complex sequence $(a_{n})$ we for all $x\ge 0$ have
\begin{equation*}
\sup_{[Re>x]} \left|\sum_{n=1}^{N} a_{n} e^{-\lambda_{n}s}\right| =\sup_{[Re=x]} \left|\sum_{n=1}^{N} a_{n} e^{-\lambda_{n}s}\right|,
\end{equation*}
which is a consequence of the modulus maximum principle (see e.g. the Phragm\'{e}n-Lindel\"{o}f Theorem from \cite[p. 137]{Besicovitch} or \cite[Lemma 1.7, \S 1.1]{Defant} for $\lambda=(\log (n))$, where the proof extends to the general case). 
\begin{proof}[Proof of Proposition \ref{BohrCahen}]
Assume, that $\Delta < \infty$ (otherwise the claim is trivial) and let $\varepsilon >0$. Then by definition of $\Delta$
\begin{equation*}
\exists N_{1} \exists j_{1} \forall N\ge N_{1}  \forall j\ge j_{1}: \frac{\log\left(\sup_{t \in \mathbf{R}}\left|\sum_{n=1}^{N}a_{n}^{j} e^{-\lambda_{n}it}\right|\right)}{\lambda_{N}} <\Delta+\varepsilon,
\end{equation*}
and so for all $N \ge N_{1}$ and $j\ge j_{1}$
\begin{equation} \label{coffee}
\sup_{t \in \mathbf{R}}\left|\sum_{n=1}^{N}a_{n}^{j} e^{-\lambda_{n}it}\right| \le e^{\lambda_{N}(\Delta+\varepsilon)}.
\end{equation}
Now we fix $M$,$N \ge N_{1}$, $j\ge j_{1}$, write for simplicity $S^{j}_{n}(it):=\sum^{n}_{k=1}a_{k}^{j} e^{-\lambda_{k}it}$ and we set $\sigma_{0}:=\Delta+2\varepsilon$. Then by Abel summation and $(\ref{coffee})$
\begin{align*}
&\left|\sum^{M}_{n=N+1} a^{j}_{n}e^{-\lambda_{n}(\sigma_{0}+it)}\right| \\ &\le \left| S^{j}_{N}(it)\right| e^{-\lambda_{N}\sigma_{0}}  + |S^{j}_{M}(it)|e^{-\lambda_{M}\sigma_{0}}+ \sum^{M-1}_{n=N} |S^{j}_{n }(it)| \left| e^{-\lambda_{n}\sigma_{0}}-e^{-\lambda_{n+1}\sigma_{0}}\right| \\ &\le e^{-\lambda_{N}\varepsilon}+e^{-\lambda_{M}\varepsilon} +\sum^{M-1}_{n=N} |S^{j}_{ n  }(it)| \left| e^{-\lambda_{n}\sigma_{0}}-e^{-\lambda_{n+1}\sigma_{0}}\right|.
\end{align*}
Since
\begin{equation*}
\left| e^{-\lambda_{n}\sigma_{0}}-e^{-\lambda_{n+1}\sigma_{0}}\right| =\left|\sigma_{0} \int_{\lambda_{n}}^{\lambda_{n+1}}e^{-\sigma_{0} x} dx \right| \le |\sigma_{0}|  \int_{\lambda_{n}}^{\lambda_{n+1}}e^{-\sigma_{0} x} dx,
\end{equation*}
we obtain
\begin{align*}
&\sum_{n=N}^{M-1} |S^{j}_{n}(it)| \left| e^{-\lambda_{n}\sigma_{0}}-e^{-\lambda_{n+1}\sigma_{0}}\right| \le |\sigma_{0}| \sum_{n=N}^{M-1}e^{\lambda_{n}(\Delta+\varepsilon)} \int_{\lambda_{n}}^{\lambda_{n+1}}e^{-\sigma_{0} x} dx\\ & \le |\sigma_{0}| \sum_{n=N}^{M-1} \int_{\lambda_{n}}^{\lambda_{n+1}}e^{-\sigma_{0} x}e^{x(\Delta+\varepsilon)} dx = |\sigma_{0}| \sum_{n=N}^{M-1} \int_{\lambda_{n}}^{\lambda_{n+1}}e^{-\varepsilon x} dx \\ &=|\sigma_{0}| \int_{\lambda_{N}}^{\lambda_{M}}e^{-\varepsilon x}dx\le |\sigma_{0}| \int_{\lambda_{N}}^{\infty}e^{-\varepsilon x}dx=|\sigma_{0}| \frac{1}{\varepsilon} e^{-\lambda_{N}\varepsilon}.
\end{align*}
So together we for all $M\ge N \ge N_{1}$ and $j\ge j_{1}$ have
\begin{equation*}
\sup_{[Re = \sigma_{0}]}\left| \sum_{n=N}^{M} a_{n}^{j} e^{-\lambda_{n}s}\right|\le e^{-\lambda_{N}\varepsilon}\left( 2+\frac{|\sigma_{0}|}{\varepsilon}\right).
\end{equation*}
Now tending $j \to \infty$ gives
\begin{equation*}
\sup_{[Re =\sigma_{0}]}\left| \sum_{n=N}^{M} a_{n}e^{-\lambda_{n}s} \right|\le e^{-\lambda_{N}\varepsilon}\left( 2+\frac{|\sigma_{0}|}{\varepsilon}\right),
\end{equation*}
which implies that $D$ converges on $[Re>\Delta]$.
Moreover for all $j\ge j_{1}$ and $N\ge N_{1}$ we have
\begin{align*}
\left| \sum_{n=1}^{\infty} (a_{n}^{j}-a_{n})e^{-\lambda_{n} s} \right|  &\le \sum_{n=1}^{N} |a_{n}^{j}-a_{n}| + \left| \sum_{n=N}^{\infty} (a_{n}^{j}-a_{n}) e^{-\lambda_{n} s}\right| \\ & =\sum_{n=1}^{N} |a_{n}^{j}-a_{n}| + \lim_{M \to \infty} \lim_{k \to \infty} \left| \sum_{n=N}^{M} (a_{n}^{j}-a_{n}^{k})e^{-\lambda_{n}s} \right| \\ &\le \sum_{n=1}^{N} |a_{n}^{j}-a_{n}|+ 2 e^{-\lambda_{N}\varepsilon}\left( 2+\frac{|\sigma_{0}|}{\varepsilon}\right),
\end{align*}
and so
\begin{equation*}
\limsup_{j \to \infty} \sup_{[Re=\sigma_{0}]}\left| \sum_{n=1}^{\infty} (a_{n}^{j}-a_{n})e^{-\lambda_{n} s} \right| \le e^{-\lambda_{N}\varepsilon}\left( 2+\frac{|\sigma_{0}|}{\varepsilon}\right),
\end{equation*}
which  proves the claim tending $N \to \infty$.
\end{proof}
\begin{corollary} \label{BohrCahenformula} Let $D=\sum a_{n} e^{-\lambda_{n}s}$ be a $\lambda$-Dirichlet series. Then
\begin{equation*}
\sigma_{u}(D)\le \limsup_{N} \frac{\log\left(\sup_{t \in \mathbf{R}} \left|\sum^{N}_{n=1} a_{n}e^{-\lambda_{n}it}\right|\right)}{\lambda_{N}}.
\end{equation*}

\end{corollary}
\begin{proof}
Defining $D_{j}=\sum_{n=1}^{j}a_{n}e^{-\lambda_{n}s}=\sum a_{n}^{j} e^{-\lambda_{n}s}$ we obtain that 
\begin{equation*}
\Delta((D_{j}))\le \limsup_{N} \frac{\log\left(\sup_{t \in \mathbf{R}} \left|\sum^{N}_{n=1} a_{n}e^{-\lambda_{n}it}\right|\right)}{\lambda_{N}}.
\end{equation*}
Indeed, if $j$, $N \in \mathbf{N}$ and $m:=\min(j,N)$, then
\begin{align*}
\frac{\log\left(\sup_{t \in \mathbf{R}} \left|\sum^{N}_{n=1} a^{j}_{n}e^{-\lambda_{n}it}\right|\right)}{\lambda_{N}} &\le \frac{\log\left(\sup_{t \in \mathbf{R}} \left|\sum^{m}_{n=1} a_{n}e^{-\lambda_{n}it}\right|\right)}{\lambda_{m}}.
\qedhere
\end{align*}
\end{proof}
\section{Main result and approximation by typical Riesz means} \label{means}
Recall that by Definition \ref{defbohrtheorem} a frequency $\lambda$ satisfies Bohr's theorem if every $D\in \mathcal{D}^{ext}_{\infty}(\lambda)$ converges uniformly on $[Re>\varepsilon]$ for all $\varepsilon>0$ or equivalently the equality
\begin{equation} \label{problem}
\sigma_{b}^{ext} = \sigma_{u}
\end{equation}
holds for all somewhere convergent $\lambda$-Dirichlet series. As already mentioned in the introduction it was a prominent question in the beginning of the 20th century for which $\lambda$'s the equality $(\ref{problem})$ holds.
The following remark shows how control of the norm  of the partial sum operator
\begin{equation*}
S_{N} \colon \mathcal{D}_{\infty}^{ext}(\lambda) \to \mathcal{D}_{\infty}(\lambda), ~~ D \mapsto \sum_{n=1}^{N}a_{n}(D) e^{-\lambda_{n}s}, ~ N \in \mathbf{N},
\end{equation*}
is linked with \eqref{problem}. 
\begin{remark} \label{finalidea} By Corollary \ref{BohrCahenformula} equality $(\ref{problem})$  holds if 
\begin{equation*}
\limsup_{N \to \infty} \frac{\log(\|S_{N}\colon \mathcal{D}_{\infty}^{ext}(\lambda)\to \mathcal{D}_{\infty}(\lambda)\|)}{\lambda_{N}}=0.
\end{equation*} 
\end{remark}

 As announced our main result gives bounds of $\|S_{N}\|$ without any assumptions on $\lambda$,  which is a sort of uniform version of Theorem 21 of \cite[p. 36]{HardyRiesz}.
\begin{theorem} \label{keylemma} For all $0<k\le1$, $N \in \mathbf{N}$ and $D=\sum a_{n} e^{-\lambda_{n}s}\in D_{\infty}^{ext}(\lambda)$ we have
\begin{equation*}
\sup_{[Re>0]}\left| \sum_{n=1}^{N} a_{n}e^{-\lambda_{n}s} \right|\le C \frac{\Gamma(k+1)}{k} \left(\frac{\lambda_{N+1}}{\lambda_{N+1}-\lambda_{N}} \right)^{k} \|D\|_{\infty},
\end{equation*}
where $C>0$ is a universal constant and $\Gamma$ denotes the Gamma function.
\end{theorem}
\begin{remark}
According to Theorem \ref{keylemma} and Corollary \ref{BohrCahenformula} the equality $\sigma_{b}^{ext}=\sigma_{u}$ holds, if there is a zero sequence $(k_{N})$ such that
\begin{equation*}
\limsup_{N \to \infty} \frac{\log\left(\frac{1}{k_{N}} \left(\frac{\lambda_{N+1}}{\lambda_{N+1}-\lambda_{N}}\right)^{k_{N}}\right)}{\lambda_{N}}=0.
\end{equation*}
\end{remark}
 In Section \ref{Bohrstheosection} we revisit the conditions $(BC)$ and $(LC)$ of Bohr and Landau (see (\ref{BC}) and (\ref{LC}) in the introduction), and show that they are sufficient for $(\ref{problem})$ by choosing suitable sequences $(k_{N})$ (Theorem \ref{BohrstheoremLandau22} and Theorem \ref{Bohrstheorem}).\\

Now let us  prepare the proof of Theorem \ref{keylemma}. We need several ingredients, and start with the following result, which is of independent interest.
\begin{proposition} \label{approx}
Let $D=\sum a_{n} e^{-\lambda_{n}s}\in D_{\infty}^{ext}(\lambda)$ with extension $f$. Then for all $k>0$ the Dirichlet polynomials
\begin{equation*}
R_{x}^{k}(D)=\sum_{\lambda_{n}<x} a_{n} \left(1-\frac{\lambda_{n}}{x}\right)^{k} e^{-\lambda_{n}s}
\end{equation*}
converge uniformly to $f$ on $[Re>\varepsilon]$ as $x\to \infty$ for all $\varepsilon>0$. Moreover, 
\begin{equation} \label{stau}
\sup_{x\ge 0} \| R_{x}^{k}(D)\|_{\infty} \le  \frac{e}{\pi}\frac{\Gamma(k+1)}{k}\|D\|_{\infty}.
\end{equation}
\end{proposition}
Proposition \ref{approx} is indicated after the proof of Theorem 41 in \cite[p. 53]{HardyRiesz} (without inequality (\ref{stau})). In the language of \cite{HardyRiesz} it states that on every smaller halfplane $[Re>\varepsilon]$ the limit functions of Dirichlet series $D\in \mathcal{D}_{\infty}(\lambda)$ are uniform limits of their typical (first) means of any order $k>0$. The proof relies on a formula of Perron (see \cite[Theorem 39, p. 50]{HardyRiesz}). We give an alternative proof of this formula (Lemma \ref{perron2}) using the Fourier inversion formula and we first deduce (\ref{stau})  from it. Then, using a Bohr-Cahen type formula for the abscissa of uniform summability by typical first means (Lemma \ref{BohrCahenforkpos}), we show that (\ref{stau}) implies the first part of Proposition \ref{approx}. Note that $\lambda$ satisfies Bohr's Theorem, if the first part of Proposition \ref{approx} is valid for $k=0$.\\

The second main ingredient for the proof of Theorem \ref{keylemma} links partial sums of a Dirichlet series to its typical means.
\begin{lemma} \label{hardy} For any choice of complex numbers $a_{1}, \ldots a_{N}$ and $0<k\le 1$ we have
\begin{equation*}
\left| \sum_{n=1}^{N} a_{n} \right| \le 3\left(\frac{1}{\lambda_{N+1}-\lambda_{N}}\right)^{k} \sup_{0\le x\le \lambda_{N+1}} \left| \sum_{\lambda_{n}<x} a_{n}(x-\lambda_{n})^{k}\right|.
\end{equation*}
\end{lemma}

Let us first show how Proposition \ref{approx} and Lemma \ref{hardy} gives Theorem \ref{keylemma} before we prove them.
\begin{proof}[Proof of Theorem \ref{keylemma}]
Let $s \in [Re>0]$.  Then by (\ref{stau}) from Proposition \ref{approx}
\begin{align*}\sup_{0\le x\le \lambda_{N+1}} \left|\sum_{\lambda_{n}<x} a_{n}e^{-\lambda_{n}s} (x-\lambda_{n})^{k} \right|&\le\lambda_{N+1}^{k} \sup_{0\le x\le \lambda_{N+1}} \left|\sum_{\lambda_{n}<x} a_{n}e^{-\lambda_{n}s} \left(1-\frac{\lambda_{n}}{x}\right)^{k} \right|\\ &\le \lambda_{N+1}^{k} \frac{e}{\pi} \frac{\Gamma(k+1)}{k}\|D\|_{\infty}.
\end{align*}
So together with Lemma \ref{hardy} we obtain
\begin{equation*}
\sup_{[Re>0]}\left| \sum_{n=1}^{N} a_{n}e^{-\lambda_{n}s} \right| \le  3\frac{e}{\pi}  \frac{\Gamma(k+1)}{k} \left(\frac{\lambda_{N+1}}{\lambda_{N+1}-\lambda_{N}} \right)^{k} \|D\|_{\infty}. \qedhere
\end{equation*}
\end{proof}
Now let us prepare the proofs of Proposition \ref{approx} and Lemma \ref{hardy}. The former relies on the following  formula of Perron from \cite[Theorem 39, p. 50]{HardyRiesz}.
\begin{lemma} \label{perron2} Let $D=\sum a_{n}e^{-\lambda_{n}s} \in \mathcal{D}_{\infty}^{ext}(\lambda)$ with extension $f$. Then for all $\varepsilon>0$ and $k\ge0$
\begin{equation} \label{alwaysperron}
\sum_{\lambda_{n}< x} a_{n}\left(1-\frac{\lambda_{n}}{x}\right)^{k}= \frac{\Gamma(k+1)}{2\pi i } \frac{1}{x^{k}}\int_{\varepsilon-i\infty}^{\varepsilon+i\infty} \frac{f(s)e^{x s}}{s^{1+k}} ds
\end{equation}
for all $x\in \mathbf{R}$ if $k>0$ and $x\notin \lambda$ whenever $k=0$.
\end{lemma}
For the sake of completeness we like to mention, that actually, if $k=0$ and $x=\lambda_{n}$ for some $n$, then (\ref{alwaysperron}) also holds true, if the integral on the right hand side is defined by its principle value (see \cite[Theorem 13, p. 12]{HardyRiesz}). This case (which is not needed in the following) is not covered by our alternative proof of Lemma \ref{perron2}, where we use the Fourier inversion formula. We need the following observation.
 
\begin{lemma} \label{l1} Let $D=\sum a_{n}e^{-\lambda_{n}s}\in \mathcal{D}(\lambda)$ with $\sigma_{c}(D)< \infty$. Then for all $\sigma>\max(0,\sigma_{c}(D))$ and $k\ge 0$ the function
\begin{equation*}
x\mapsto e^{-\sigma x}\sum_{\lambda_{n}<x}a_{n}\left(x-\lambda_{n}\right)^{k}
\end{equation*}
is in $L_{1}(\mathbf{R})$.
\end{lemma}
\begin{proof} Fix $\sigma>\max(0,\sigma_{c}(D))$ and write $A^{k}(t):=\sum_{\lambda_{n}<t}a_{n}\left(t-\lambda_{n}\right)^{k}$. Let first $k=0$. Then by Abel summation for all $t>0$
\begin{equation}\label{hunger}
A^{0}(t)=e^{t\sigma} \sum_{\lambda_{n}<t} a_{n}e^{-\sigma\lambda_{n}}-\sigma \int_{0}^{t} e^{\sigma y}  \sum_{\lambda_{n}<y} a_{n}e^{-\sigma\lambda_{n}} dy.
\end{equation}
 Let $\varepsilon >0$. Since $D$ converges at $\sigma$ we for all $0<t\le x$ obtain multiplying with $e^{- (\sigma+\varepsilon) x}$
\begin{align} \label{plus5}
|A^{0}(t)| e^{-(\sigma+\varepsilon) x}\le C(\sigma)e^{-\varepsilon x}+\sigma C(\sigma) e^{-\varepsilon x}\int_{0}^{\infty}e^{-\sigma u} du= 2 C(\sigma)e^{-\varepsilon x}.
\end{align}
In particular $|A^{0}(x)|e^{-(\sigma+\varepsilon) x}\le 2C(\sigma)e^{-\varepsilon x}$ for all $x\in \mathbf{R}$ and so $A^{0}e^{-(\sigma+\varepsilon) \cdot}\in L_{1}(\mathbf{R})$. Now let $k>0$. By Abel summation for all $x\ge 0$ we have
\begin{equation} \label{plus6}
A^{k}(x)=\sum_{\lambda_{n}<x}a_{n}(x-\lambda_{n})^{k}=k \int_{0}^{x} (x-t)^{k-1} A^{0}(t)~ dt,
\end{equation}
which is taken from  \cite[Chapter IV, \S 2, p. 21]{HardyRiesz}. Again by multiplying with $e^{-(\sigma+\varepsilon) x}$ we obtain using $(\ref{plus5})$
\begin{align*}
e^{-(\sigma+\varepsilon) x} |A^{k}(x)|\le 2C(\sigma)e^{-\varepsilon x} k \int_{0}^{x} (x-t)^{k-1} dt=2C(\sigma)e^{-\varepsilon x} x^{k},
\end{align*}
and so $e^{-(\sigma+\varepsilon) \cdot} A^{k} \in L_{1}(\mathbf{R})$.
\end{proof}

\begin{proof}[Proof of Lemma \ref{perron2}] Let first $\sigma>\max(0,\sigma_{c}(D))$. Then for all $s \in [Re>\sigma]$ (see \cite[Theorem 24, p. 39]{HardyRiesz}, where $k>0$ is considered, but the case $k=0$ follows in the same way applying Abel summation)
\begin{equation*}
f(s)=\frac{1}{\Gamma(k+1)} \int_{0}^{\infty} s^{k+1}e^{-st}\sum_{\lambda_{n}<t} a_{n}(t-\lambda_{n})^{k}  dt
\end{equation*}
and so
\begin{equation*}
\frac{\Gamma(k+1)}{2\pi i}\int_{2\sigma-i\infty}^{2\sigma+i\infty} \frac{f(s)e^{x s}}{s^{k+1}} ds=\frac{1}{2\pi i}  \int_{2\sigma-i\infty}^{2\sigma+i\infty} \int_{0}^{\infty} e^{s(x-t)} \sum_{\lambda_{n}<t} a_{n}(t-\lambda_{n})^{k} dt ds.
\end{equation*}
Now again for simplicity we write $A^{k}(t):=\sum_{\lambda_{n}<t} a_{n}(t-\lambda_{n})^{k}$, which is a differentiable function on $\mathbf{R}$, if $k>0$. The function $A^{0}$ is not differentiable at $\lambda_{n}$ for all $n$, but elsewhere else. Moreover, by Lemma \ref{l1} we know that  $A^{k} e^{-2\sigma \cdot} \in L_{1}(\mathbf{R})$ for all $k\ge 0$. Now, denoting by $\mathcal{F}_{L_{1}(\mathbf{R})}$ the  Fourier transform on $L_{1}(\mathbf{R})$, the Fourier inversion formula (see \cite[\S 1.2]{HelsonBook}) gives
\begin{align*}
&\frac{\Gamma(k+1)}{2\pi i}\int_{2\sigma-i\infty}^{2\sigma+i\infty} \frac{f(s)e^{x s}}{s^{k+1}} ds=\frac{1}{2\pi i}  \int_{2\sigma-i\infty}^{2\sigma+i\infty} \int_{0}^{\infty} e^{s(x-t)} A^{k}(t) dt ds\\ &=\frac{1}{2 \pi} \int_{-\infty}^{\infty} e^{x(2\sigma+iy)} \int_{0}^{\infty} A^{k}(t) e^{-t(2\sigma+iy)} dt dy\\ &= e^{2x \sigma} \int_{-\infty}^{\infty} \mathcal{F}_{L_{1}(\mathbf{R})}\left(A^{k}e^{-2\sigma \cdot }\right)(y) e^{iyx} dy\\ &=e^{2x \sigma } \mathcal{F}_{L_{1}(\mathbf{R})}(\mathcal{F}_{L_{1}(\mathbf{R})}(A^{k}e^{-2 \sigma \cdot}))(-x)=e^{2x\sigma} A^{k}(x)e^{-2x\sigma}=A^{k}(x),
\end{align*}
where $x \notin \lambda$ if $k=0$ and $x$ arbitrary if $k>0$.
This implies (\ref{alwaysperron}) for $\varepsilon=2\sigma$, where $\sigma>\max(0,\sigma_{c}(D))$. For $\varepsilon>0$ arbitrary by Cauchy's integral theorem (and boundedness of $f$) we have
\begin{equation*}
\int_{\varepsilon-i\infty}^{\varepsilon+i\infty} \frac{f(s)e^{x s}}{s^{1+k}} ds=\int_{2\big(\varepsilon+\max\left(0,\sigma_{c}(D)\right)\big)-i\infty}^{2\big(\varepsilon+\max(0,\sigma_{c}(D))\big)+i\infty} \frac{f(s)e^{x s}}{s^{1+k}} ds,
\end{equation*}
which finishes the proof. \qedhere
\end{proof}

For the proof of Proposition \ref{approx} we still need to verify the following Bohr-Cahen type formula for the abscissa $\sigma_{u}^{k}(D)$ of uniform summability by typical first means of order $k$, which is defined to be the infimum of all $\sigma\in \mathbf{R}$ such that $(R_{x}^{k}(D))$ converges uniformly on $[Re>\sigma]$ as $x\to \infty$.
\begin{lemma} \label{BohrCahenforkpos} Let $0< k\le 
1$ and $D\in \mathcal{D}(\lambda)$. Then 
\begin{equation} \label{blau}
\sigma_{u}^{k}(D)\le \limsup_{x\to \infty} \frac{\log(\|R_{x}^{k}(D)\|_{\infty})}{x},
\end{equation}
where equality holds whenever $\sigma_{u}^{k}(D)$ is non negative.
\end{lemma}

\begin{proof} Let $L$ denote the right hand side of (\ref{blau}). We first show (\ref{blau}) and assume that $L<\infty$, since otherwise the claim is trivial.  Hence, fixing $\varepsilon>0$, there is a constant $C$ such that for all $x$
\begin{equation} \label{blau1}
\|R_{x}^{k}(D)\|_{\infty}\le Ce^{(L+\varepsilon)x}.
\end{equation}
Let $u:=L+3\varepsilon$ and we claim that $(R^{k}_{x}(D))$ converges uniformly on $[Re=u]$, tending $x\to \infty$. For $w\in \mathbf{C}$ for notational simplicity we write $A^{k}_{w}(t)=\sum_{\lambda_{n}<t} a_{n}e^{-w \lambda_{n}} (t-\lambda_{n})^{k}$, where $0\le k \le 1$.
Following the lines of \cite[\S VI.3 b), p. 42]{HardyRiesz} by Abel's summation we for all $s$, $w\in \mathbf{C}$ obtain
 \begin{equation}\label{blau2}
 R_{x}^{k}(D)(s+w)=e^{-sx} R_{x}^{k}(D)(w)-\frac{1}{x^k}\int_{0}^{x}A^{0}_{w}(t) \frac{d}{dt}  \big((e^{-st}-e^{-sx})(x-t)^{k}\big) dt.
 \end{equation}
 In particular, with the choice $s=u$ and $w=i\tau$, where $\tau \in \mathbf{R}$, by (\ref{blau1}) the first term on the right hand side of (\ref{blau2}) vanishes whenever $x\to \infty$. Applying integration by parts we obtain
\begin{equation} \label{finally}
-\int_{0}^{x}A^{0}_{w}(t) \frac{d}{dt}  \big((e^{-st}-e^{-sx})(x-t)^{k}\big) dt=\int_{0}^{x} A^{1}_{i\tau}(t) \frac{d^{2}}{dt^{2}} \big((e^{-ut}-e^{-ux})(x-t)^{k}\big) dt,
\end{equation}
where the second derivative appearing is given by
\begin{align*}
&u^{2}e^{-ut}(x-t)^{k}+2kue^{-ut}(x-t)^{k-1}+k(k-1)(e^{-ut}-e^{-ux})(x-t)^{k-2}\\ &=:g_{1}(t)+g_{2}(t)+g_{3}(t).
\end{align*}
If $k=1$, then $g_{3}$ vanishes and by assumption we have 
\begin{equation*}
|A_{i\tau}^{1}(t)|=|tR_{t}^{1}(D)(i\tau)|\le C_{1}e^{(L+2\varepsilon)t}
\end{equation*}
with some constant $C_{1}=C_{1}(\varepsilon)$. If $0<k<1$, then by \cite[Lemma 6, p. 27]{HardyRiesz} 
\begin{equation} \label{blau111}
\Gamma(k+1)\Gamma(1-k)A^{1}_{i\tau}(t)= \int_{0}^{t} A^{k}_{i\tau}(y) (t-y)^{-k} dy,
\end{equation}
and so for all $t>0$ with $C_{2}=C_{2}(k)=\big (\Gamma(k+1)\Gamma(1-k)(1-k)\big)^{-1}$
\begin{equation*}
|A^{1}_{i\tau}(t)|\le C_{2}t^{1-k}  \sup_{0<y<t} |A^{k}_{i\tau}(y)|=C_{2}t^{1-k}  \sup_{0<y<t} |y^{k}R_{y}^{k}(i\tau)| \le C_{3}e^{(L+2\varepsilon) t},
\end{equation*}
where $C_{3}=C_{3}(k,\varepsilon)$. With $C_{4}:=\max(C_{1}, C_{3})$ for $0<k\le 1$ we have
\begin{align*}
\left|\frac{1}{x^{k}} \int_{0}^{x} A^{1}_{t}(i\tau)g_{2}(t)dt\right|&\le C_{4} \frac{2ku}{x^{k}} \left(\int_{0}^{\frac{x}{2}} e^{-\varepsilon t}(x-t)^{k-1} dt+ \int_{\frac{x}{2}}^{x}e^{-\varepsilon t} (x-t)^{k-1} dt \right) \\ &\le C_{4}\frac{2ku}{x^{k}} \left( \left(\frac{x}{2}\right)^{k-1} \int_{0}^{\infty}e^{-\varepsilon t} dt+ e^{-\varepsilon \frac{x}{2}} \int_{\frac{x}{2}}^{x} (x-t)^{k-1} dt \right) \\ &=C_{4}\frac{2ku}{x^{k}} \left( \left(\frac{x}{2}\right)^{k-1} \frac{1}{\varepsilon}+e^{-\varepsilon \frac{x}{2}} \frac{1}{k} \left(\frac{x}{2}\right)^{k} \right)\le C_{5} \left(\frac{1}{x}+e^{-\varepsilon \frac{x}{2}} \right),
\end{align*}
which tends to zero uniformly in $\tau$ as $x\to \infty$.  Analogously, using $|e^{-ut}-e^{-ux}|\le u (x-t)e^{-tu}$, we obtain
\begin{align*}
\left|\frac{1}{x^{k}}\int_{0}^{x} A^{1}_{t}(i\tau) g_{3}(t) dt \right|&\le C_{4}\frac{k(1-k)}{x^{k}} \int_{0}^{x} e^{-\varepsilon t}u (x-t)(x-t)^{k-2} dt\\ &= C_{4}\frac{k(1-k)u}{x^{k}} \int_{0}^{x} e^{-\varepsilon t} (x-t)^{k-1} dt,
\end{align*}
which also vanishes uniformly in $\tau$ tending $x\to \infty$. It remains to consider the integral with $g_{1}$. The dominated convergence theorem implies for all $\tau$
\begin{equation*}
\lim_{x\to \infty} \frac{1}{x^{k}} \int_{0}^{x} u^{2}e^{-ut}(x-t)^{k} A^{1}_{t}(i\tau) dt=\int_{0}^{\infty}u^{2}e^{-ut} A^{1}_{t}(i\tau) dt,
\end{equation*}
and we claim that the convergence is uniform in $\tau$. Indeed, we have
\begin{align*}
&\left| \frac{1}{x^{k}} \int_{0}^{x} u^{2}e^{-ut}(x-t)^{k} A^{1}_{t}(i\tau) dt-\int_{0}^{\infty} u^{2}e^{-ut}A^{1}_{t}(i\tau) dt\right| \\ &=\left| \int_{0}^{\infty} u^{2}e^{-ut} A^{1}_{t}(i\tau) \left(\chi_{(0,x)}(t) \left(1-\frac{t}{x}\right)^{k}-1\right) dt \right|\\ &\le C_{4} \int_{0}^{\infty} u^{2}e^{-\varepsilon t} \left(1-\chi_{(0,x)}(t) \left(1-\frac{t}{x}\right)^{k}\right) dt.
\end{align*}
Now, since the function $h_{x}(t):=\left(1-\chi_{(0,x)}(t) \left(1-\frac{t}{x}\right)^{k}\right)$ vanishes as $x \to \infty$ and $|h_{x}(t)|\le 1$, by the dominated convergence theorem we obtain
\begin{equation} \label{limitlimit}
\lim_{x\to \infty} \frac{1}{x^{k}} \int_{0}^{x} u^{2}e^{-ut}(x-t)^{k} A^{1}_{t}(i\tau) dt=\int_{0}^{\infty}u^{2}e^{-ut} A^{1}_{t}(i\tau) dt
\end{equation}
uniformly in $\tau$. In summary, the sequence $(R_{x}^{k}(D))$ on $[Re=u]$ converges uniformly to (\ref{limitlimit}) tending $x\to \infty$, which proves (\ref{blau}). Now assume that $\sigma_{u}^{k}(D)\ge 0$. Let $\varepsilon>0$ and $u:=\sigma_{u}^{k}(D)+\varepsilon$. Then $(R_{x}^{k}(D))$ converges uniformly on $[Re=u]$, and so $\sup_{x} \sup_{[Re=u]} |R_{x}^{k}(D)(s)|=:C_{6}<\infty$. Let first $0<k<1$. Then by (\ref{blau2}) with the choice $s=-u$ and $w=u+i\tau$, following the calculation we used before (in particular using the adapted variant of (\ref{blau111})) we obtain
\begin{align*}
&|R_{x}^{k}(D)(i\tau)| \\ & \le C_{6}e^{ux}+\frac{1}{x^{k}} \frac{1}{\Gamma(1+k)\Gamma(1-k)}\int_{0}^{x} \int_{0}^{t} |A^{k}_{w}(y)| (t-y)^{-k} dy \left|\frac{d^{2}}{dt^{2}} \big((e^{ut}-e^{ux})(x-t)^{k} \big)\right| dt\\ &\le C_{6}\left(e^{ux}+\frac{1}{x^{k}}\frac{1}{\Gamma(1+k)\Gamma(1-k)}\int_{0}^{x} \int_{0}^{t} y^{k} (t-y)^{-k} dy \left|\frac{d^{2}}{dt^{2}} \big((e^{ut}-e^{ux})(x-t)^{k}\big) \right| dt \right) 
\end{align*}
By substitution with $v=\frac{y}{t}$ we for every $\alpha$, $\beta$, $t>0$ obtain
\begin{equation} \label{gammageneral}
\frac{1}{t^{\alpha+\beta}}\int_{0}^{t}y^{\alpha}(t-y)^{\beta-1} dy=\int_{0}^{1}v^{\alpha}(1-v)^{\beta-1} dv=\frac{\Gamma(\alpha+1)\Gamma(\beta)}{\Gamma(\alpha+\beta+1)}.
\end{equation}
In particular, if $0<k<1$, then the choice $\alpha=k$ and $\beta=1-k$ gives
\begin{equation} \label{k=1}
\int_{0}^{t}y^{k} (t-y)^{-k} dy= \Gamma(k+1)\Gamma(1-k) t.
\end{equation}
Using that, we continue estimating
\begin{align*}
|R_{x}^{k}(D)(i\tau)| &\le C_{6} \left( e^{ux} + \frac{1}{x^{k}} \int_{0}^{x} t\left|\frac{d^{2}}{dt^{2}} \big((e^{ut}-e^{ux})(x-t)^{k}\big)\right| dt \right)\\ &\le 
C_{6}e^{ux} \left(1+ \frac{1}{x^{k-1}} \int_{0}^{x} u^{2}(x-t)^{k}+ (2ku+k(1-k)u)(x-t)^{k-1} dt \right)\\ &\le C_{7}e^{ux}(1+x^{2}+x).
\end{align*}
Hence for all $\varepsilon>0$
\begin{equation*}
\limsup_{x\to \infty} \frac{\log(\|R_{x}^{k}(D)\|_{\infty})}{x}\le u+\limsup_{x\to \infty} \frac{\log(C_{7}(1+x^{2}+x))}{x}=\sigma_{u}^{k}(D)+\varepsilon,
\end{equation*}
 and so $L\le \sigma_{u}^{k}(D)$, if $0<k<1$. The remaining case $k=1$ follows in a similiar way. Indeed, (\ref{blau2})  with the choice $s=-u$ and $w=u+i\tau$ together with the adapted variant of  (\ref{finally}) leads to
 \begin{align*}
 |R_{x}^{1}(D)(i\tau)|\le C_{6} \left(e^{ux}+ \frac{1}{x}  \int_{0}^{x}t\left(u^{2}e^{ut}(x-t)+2ue^{ut}\right) dt \right)\le C_{8}e^{ux}(1+x^{2}),
 \end{align*}
which finally completes the proof.
\end{proof}

\begin{proof}[Proof of Proposition \ref{approx}]
First we prove the stated inequality. Let $x\ge0$ and $\varepsilon=\frac{1}{x}$. Then applying for fixed $s_{0}\in [Re>0]$ Lemma \ref{perron2} to the translated Dirichlet series $D_{s_{0}}(s)=\sum a_{n} e^{-\lambda_{n}s_{0}} e^{-\lambda_{n}s}$ with extension $f_{s_{0}}(s):=f(s+s_{0})$  we obtain  \begin{align*} &\sup_{s_{0}\in [Re>0]} \left| \sum_{\lambda_{n}<x} a_{n}\left(1-\frac{\lambda_{n}}{x}\right)^{k} e^{-\lambda_{n}s_{0}} \right| \le \|D\|_{\infty}\frac{1}{x^{k}} \frac{\Gamma(k+1)}{2\pi} e^{x \varepsilon}2\int_{0}^{\infty} \frac{1}{|\varepsilon+it|^{1+k}} dt\\ &\le  \|D\|_{\infty} \frac{1}{x^{k}}\frac{\Gamma(k+1)}{\pi} e \frac{1}{k} x^{k} = \|D\|_{\infty} \frac{e}{\pi} \frac{\Gamma(k+1)}{k} . \end{align*}
Now having (\ref{stau}), Lemma \ref{BohrCahenforkpos} implies that for $0<k\le 1$ the Dirichlet polynomials $(R_{x}^{k}(D))$ converge uniformly on $[Re>\varepsilon]$ as $x\to \infty$ for every $\varepsilon>0$. Then $h(s):=\lim_{x\to \infty} R_{x}^{k}(D)(s)$ defines a holomorphic function (see e.g. \cite[Theorem 27, p. 44]{HardyRiesz}) and, since $D$ converges at some $s_{0}=\sigma_{0}+i\tau_{0}$, $h$ coincides with $f$ on $[Re> \sigma_{0}]$ (see also \cite[Theorem 16, p. 29]{HardyRiesz}). Hence $h=f$ and $(R_{x}^{k}(D))$ converges uniformly to $f$ on every smaller halfplane contained in $[Re>0]$. So the claim holds for $0<k\le 1$. Let $k>1$ and write $k=l+k^{\prime}$, where $0<k^{\prime}\le 1$ and $l \in \mathbf{N}$. Then by \cite[Lemma 6, p. 27]{HardyRiesz}, or by direct inspection
\begin{equation}
R_{x}^{k}(D)(s)= \frac{\Gamma(k+1)}{\Gamma(k^{\prime}+1)\Gamma(l)} \frac{1}{x^{k}}\int_{0}^{x} \left(\sum_{\lambda_{n}<t} a_{n}e^{-\lambda_{n}s}(t-\lambda_{n})^{k^{\prime}}\right) (x-t)^{l-1} dt.
\end{equation}
Now by (\ref{gammageneral}) with $\alpha=k^{\prime}$ and $\beta=l$ we have
\begin{equation} \label{blau4}
\frac{\Gamma(k^{\prime}+1)\Gamma(l)}{\Gamma(k+1)}=\frac{1}{x^{k}}\int_{0}^{x}t^{k^{\prime}}(x-t)^{l-1} dt.
\end{equation}
 Then writing $C:=\frac{\Gamma(k+1)}{\Gamma(k^{\prime}+1)\Gamma(l)}$ and using (\ref{blau4}) we obtain for every $s\in [Re>0]$
 \begin{align*}
 &R_{x}^{k}(D)(s)-f(s)=R_{x}^{k}(D)(s)-Cf(s) \frac{1}{x^{k}}\int_{0}^{x}t^{k^{\prime}}(x-t)^{l-1} dt\\ &= \frac{C}{x^{k}}\int_{0}^{x} (x-t)^{l-1} t^{k^{\prime}} \left(R_{t}^{k^{\prime}}(D)(s)-f(s)\right) dt.
 \end{align*}
 So, given $\varepsilon>0$ and $u>0$, choose $x_{0}$ such that $|R_{t}^{k^{\prime}}(D)(s)-f(s)|\le \varepsilon$ for all $t>x_{0}$ and $s\in [Re>u]$. Then for all $x>x_{0}$ and $s\in [Re>u]$ we have \begin{align*}
&|R_{x}^{k}(D)(s)-f(s)|\\ &\le \sup_{y\ge 0} \|R_{y}^{k^{\prime}}(D)-f\|_{\infty} \frac{C}{x^{k}} \int_{0}^{x_{0}}(x-t)^{l-1} t^{k^{\prime}} dt +\varepsilon \frac{C}{x^{k}}\int_{x_{0}}^{x} (x-t)^{l-1} t^{k^{\prime}} dt \\ &  \le 2C\|D\|_{\infty}\frac{x_{0}^{k^{\prime}}}{x^{k}} \frac{2x^{l}}{l}+\varepsilon\frac{Cx^{k^{\prime}}}{x^{k}}\frac{x^{l}}{l}\le \frac{4C}{l} \left(\|D\|_{\infty}\left(\frac{x_{0}}{x}\right)^{k^{\prime}}+\varepsilon\right) ,
\end{align*}
 which finishes the proof tending $x\to \infty$. \qedhere
\end{proof}
Proposition \ref{approx} gives a direct link to the theory of almost periodic functions on $\mathbf{R}$ and proves that the space $(\mathcal{D}_{\infty}^{ext}(\lambda), \|\cdot\|_{\infty})$ actually is a normed space. Recall that by definition a continuous function $f \colon \mathbf{R} \to \mathbf{C}$ is called (uniformly) almost periodic, if to every $\varepsilon>0$ there is a number $l>0$ such that for all intervals $I\subset \mathbf{R}$ with $|I|=l$ there is a translation number $\tau \in I$ such that $\sup_{x\in \mathbf{R}} |f(x+\tau)-f(x)|\le \varepsilon$ (see \cite{Besicovitch} for more information). Then by a result of Bohr a bounded and continuous function $f$ is almost periodic if and only if it is the uniform limit of trigonometric polynomials on $\mathbf{R}$, which are of the form $p(t):=\sum_{n=1}^{N} a_{n}e^{-itx_{n}}$, where $x_{n}\in \mathbf{R}$ (see e.g. \cite[\S 1.5.2.2, Theorem 1.5.5]{QQ}). In particular, the Dirichlet polynomials $R_{x}^{k}(D)$ stated in Proposition \ref{approx} considered as functions on vertical lines $[Re=\sigma]$ are almost periodic.
\begin{corollary} \label{almost periodic} If $D=\sum a_{n} e^{-\lambda_{n}s}\in D_{\infty}^{ext}(\lambda)$ with extension $f$, then the function $f_{\sigma}(t):=f(\sigma+it)\colon \mathbf{R} \to  \mathbf{C}$ is almost periodic and
\begin{equation*}
a_{n}=\lim_{T\to \infty} \frac{1}{2T} \int_{-T}^{T} f(\sigma+it) e^{(\sigma+it)\lambda_{n}}
\end{equation*}
for all $\sigma >0$. In particular, $\sup_{n \in \mathbf{N}} |a_{n}|\le \|D\|_{\infty}$ and $\mathcal{D}_{\infty}^{ext}(\lambda)$ is a normed space.
\end{corollary}
\begin{proof}
Since on $[Re=\sigma]$ the limit function $f$ is  the uniform limit of $R^{1}_{x}(D)$ tending $x\to \infty$ (Proposition \ref{approx}), $f_{\sigma}$ is almost periodic. Then by \cite[Chapter I, \S 3.11, p. 21]{Besicovitch} we have
\begin{align*}
&\lim_{T\to \infty} \frac{1}{2T} \int_{-T}^{T} f(\sigma+it) e^{(\sigma+it)\lambda_{n}}=\lim_{x \to \infty}\lim_{T\to \infty} \frac{1}{2T} \int_{-T}^{T} R_{x}^{1}(\sigma+it) e^{(\sigma+it)\lambda_{n}} dt\\ & =\lim_{x\to \infty} a_{n}\left(1-\frac{\lambda_{n}}{x}\right)=a_{n}
\end{align*}
for all $\sigma>0$ and so $|a_{n}|\le \|D\|_{\infty}$.
\end{proof}
Another property of almost periodic functions is that they allow a unique continuous extension to the Bohr compactification $\overline{\mathbf{R}}$ of $\mathbf{R}$ (see \cite[\S 1.5.2.2, Theorem 1.5.5]{QQ}). In particular, the monomials $e^{-i\lambda_{n}\cdot}$ extend uniquely to characters on $\overline{\mathbf{R}}$. We like to mention that this observation led to an $\mathcal{H}_{p}$-theory of general Dirichlet series (see \cite{DefantSchoolmann}) naturally containing and extending the $\mathcal{H}_{p}$-theory of ordinary Dirichlet series invented by Bayart in \cite{Bayart}.\\

To complete our proof of Theorem \ref{keylemma}, it remains to verify Lemma \ref{hardy}.
\begin{proof}[Proof of Lemma \ref{hardy}]
We again use $(\ref{plus6})$ and obtain for all $N\in \mathbf{N}$
\begin{align*}
&(\lambda_{N+1}-\lambda_{N})^{k} \left|\sum_{n=1}^{N}a_{n} \right|  =\left|\sum_{n=1}^{N}a_{n} \right| \int_{\lambda_{N}}^{\lambda_{N+1}} k(\lambda_{N+1}-t)^{k-1} dt \\&=
\left| \int_{\lambda_{N}}^{\lambda_{N+1}} \left(\sum_{\lambda_{n}<t}a_{n} \right) k(\lambda_{N+1}-t)^{k-1} dt\right|\\&= \left|\int_{0}^{\lambda_{N+1}} \left(\sum_{\lambda_{n}<t}a_{n} \right) k(\lambda_{N+1}-t)^{k-1} dt- \int_{0}^{\lambda_{N}} \left( \sum_{\lambda_{n}<t} a_{n} \right) k(\lambda_{N+1}-t)^{k-1} dt \right| \\ &\le \left|\sum_{n=1}^{N} a_{n}(\lambda_{n+1}-\lambda_{n})^{k} \right|+ \left|\int_{0}^{\lambda_{N}} \left( \sum_{\lambda_{n}<t} a_{n} \right) k(\lambda_{N+1}-t)^{k-1} dt\right|\\ &
\le \sup_{0\le x\le \lambda_{N+1}} \left| \sum_{\lambda_{n}<x} a_{n}(x-\lambda_{n})^{k}\right|+ \left|\int_{0}^{\lambda_{N}} \left( \sum_{\lambda_{n}<t} a_{n} \right) k(\lambda_{N+1}-t)^{k-1} dt\right|.
\end{align*}

Now, applying \cite[Lemma 7, p. 28]{HardyRiesz} to the real and imaginary parts of the integral we obtain
\begin{equation*}
\left|\int_{0}^{\lambda_{N}} \left( \sum_{\lambda_{n}<t} a_{n} \right) k (\lambda_{n+1}-t)^{k-1} dt\right|\le 2\sup_{0\le x\le \lambda_{N}} \left|\sum_{\lambda_{n}<x} a_{n} (x-\lambda_{n})^{k} \right|,
\end{equation*} 
and so
\begin{equation*}
\left|\sum_{n=1}^{N}a_{n} \right| \le 3\left(\frac{1}{\lambda_{N+1}-\lambda_{N}}\right)^{k}\sup_{0\le x\le \lambda_{N+1}} \left|\sum_{\lambda_{n}<x} a_{n} (x-\lambda_{n})^{k} \right|.
\qedhere
\end{equation*}
\end{proof}

\section{On Bohr's theorem} \label{Bohrstheosection}
Now we apply our main result Theorem \ref{keylemma} to prove quantitative variants of Bohr's theorem for certain classes of $\lambda$'s (including $(BC)$ and $(LC)$) by giving bounds for $\|S_{N}\|$, $N\in \mathbf{N}$. Observe that by Corollary \ref{almost periodic} we always have the trivial bound $\|S_{N}\|\le N$. Hence by Remark \ref{finalidea} $\lambda$ satisfies Bohr's theorem (or equivalently equality $(\ref{problem})$ holds), if
\begin{equation} \label{uiui}
L(\lambda):=\limsup_{N \to \infty} \frac{\log(N)}{\lambda_{N}}=0.
\end{equation}
For instance $\lambda=(n)=(0,1,2,\ldots)$ fulfils $L((n))=0$ and we (again) see as a consequence that for power series we can not distinguish between uniform convergence and boundedness of the limit function up to $\varepsilon$. \\

We like to mention that the number $L(\lambda)$ also has a geometric meaning.
Bohr shows in \cite[ \S 3, Hilfssatz 3, Hilfssatz 2]{Bohr2} that
\begin{equation*}
L(\lambda)=\sigma_{c}\left(\sum e^{-\lambda_{n}s} \right)=\sup_{D \in \mathcal{D}(\lambda)} \sigma_{a}(D)-\sigma_{c}(D),
\end{equation*}
where the latter is the maximal width of the so called strip of pointwise and not absolutely convergence.

\begin{remark} \label{connections} We summarize relations of $(BC)$, $(LC)$ and '$L(\lambda)<\infty$'.
\begin{enumerate}
\item[1)] $(BC)$ implies $L(\lambda)<\infty$ and $(LC)$.
\item[2)] $(LC)$ and $L(\lambda)<\infty$ does not necessarily imply  $(BC)$.
\item[3)]$L(\lambda)<\infty$ does not necessarily imply nor $(LC)$ and so neither $(BC)$.
\item[4)]$(LC)$ does not necessarily imply nor  $L(\lambda)<\infty$ and so neither $(BC)$.

\end{enumerate}
\end{remark}
\begin{proof}
1) The implication $(BC) \Rightarrow (LC)$ is clear and the fact that $L(\lambda)<\infty$, if $\lambda \in (BC)$, is done in \cite[\S 3, Hilfssatz 4]{Bohr2}. 2) Take $\lambda$ defined by $\lambda_{2n}=n+e^{-n^2}$ and $\lambda_{2n+1}=n$. Then $(LC)$ is satisfied with $L(\lambda)=0$, but $\lambda$ fails for $(BC)$. 3) Define $\lambda_{2n}=n+e^{-e^{n^{2}}}$ and $\lambda_{2n-1}=n$. Then $L(\lambda)=0$ and $\lambda$ does not satisfy $(LC)$. 4) Consider $\lambda:=(\sqrt{\log(n)})$. Then $L(\lambda)=+\infty$ (and so $(BC)$ fails) but $(LC)$ is satisfied: We claim that $\lambda_{n+1}-\lambda_{n}\ge C e^{-2\lambda_{n}^{2}}$ for some $C$, that is
$\sqrt{\log(n+1)}-\sqrt{\log(n)} \ge C n^{-2}$. We have 
\begin{align*}
n^{2} \sqrt{\log(n+1)}-\sqrt{\log(n)}&=\frac{n^{2}\log\left(1+\frac{1}{n}\right)}{\sqrt{\log(n+1)}+\sqrt{\log(n)}}\\ &\ge \log\left(\left(1+\frac{1}{n}\right)^{n} \right) \frac{n}{2\sqrt{\log(n+1)}}\to +\infty. 
\end{align*}
Since to every $\delta>0$ there is a constant $C=C(\delta)$ such that $e^{-2\lambda_{n}^{2}}\ge C(\delta) e^{-e^{\delta\lambda_{n}}},$ the claim follows.
\end{proof}
\subsection{Landau's condition}
\begin{theorem} \label{BohrstheoremLandau22}
Let $(LC)$ hold for $\lambda$. Then for all $\delta>0$ there is $C=C(\delta, \lambda)$ such that for all $D=\sum a_{n} e^{-\lambda_{n}s} \in \mathcal{D}^{ext}_{\infty}(\lambda)$ and all $N \in \mathbf{N}$ we have 
\begin{equation*}
\sup_{[Re>0]}\left|\sum_{n=1}^{N} a_{n}e^{-\lambda_{n}s}\right|\le Ce^{\delta \lambda_{N}} \|D\|_{\infty}.
\end{equation*}
In particular, $\lambda$ satisfies Bohr's theorem.
\end{theorem}
\begin{proof} W.l.o.g we may assume that $\lambda_{n+1}-\lambda_{n}\le 1$ for all n. Indeed, suppose that $\lambda_{n+1}-\lambda_{n}>1$ for some $n$,  and let $l$ be the smallest natural number such that $l\ge \lambda_{n+1}-\lambda_{n}$. Then we define a new frequency $\lambda^{2}$ by adding more numbers to $\lambda$, so that $\lambda^{2}$ satisfies $(LC)$ and $\sup_{n} \lambda_{n+1}^{2}-\lambda_{n}^{2}\le 1$. We follow two steps. First, whenever $l\ge 3$, we add $\lambda_{n}+1, \lambda_{n}+2,\ldots , \lambda_{n}+l-2$ to $\lambda$. Since $1<|\lambda_{n+1}-(\lambda_{n}+l-2)|\le 2$, this procedure gives a new frequency, say $\lambda^{1}$, such that $\sup_{n} \lambda_{n+1}^{1}-\lambda_{n}^{1}\le 2$. Moreover, $\lambda^{1}$ has $(LC)$, since $\lambda$ satisfies $(LC)$ by assumption. Now, we add more numbers to $\lambda^{1}$ to obtain the announced frequency $\lambda^{2}$. If $\lambda_{n+1}^{1}-\lambda_{n}^{1}>1$ for some $n$, then we add the number $w:=\frac{\lambda^{1}_{n+1}+\lambda_{n}^{1}}{2}$ to $\lambda^{1}$. In this way the new frequency $\lambda^{2}$ fulfils $\sup_{n} \lambda_{n+1}^{2}-\lambda_{n}^{2}\le 1$ and $(LC)$, which follows from the observation $\lambda^{1}_{n+1}-w=w-\lambda^{1}_{n}$.  Now let $\delta>0$ and set $k_{N}=e^{-\delta \lambda_{N}}$. Then by Theorem \ref{keylemma} and assuming $\lambda_{n+1}-\lambda_{n}\le 1$ for all n we obtain
\begin{align*}
\|S_{N}\|\le C \frac{\Gamma(k_{N}+1)}{k_{N}} \left(\frac{\lambda_{N+1}}{\lambda_{N+1}-\lambda_{N}} \right)^{k_{N}}\le C_{1} e^{\delta \lambda_{N}} (1+\lambda_{N})^{e^{-\lambda_{N} \delta}} \le C_{2} e^{\delta \lambda_{N}}.
\end{align*}
Now Corollary \ref{BohrCahenformula} gives $\sigma_{u}(D)\le \delta$ for all $\delta>0$ and so $\sigma_{u}(D)\le 0$.
\end{proof}

\begin{corollary} \label{Bohrstheorem2} Let $\lambda$ satisfy $(LC)$. Then to every $\sigma>0$ there is a constant $C_{1}=C_{1}(\sigma, \lambda)$ such that  for all $N \in \mathbf{N}$ and $D=\sum a_{n}e^{-\lambda_{n}s}\in \mathcal{D}_{\infty}^{ext}(\lambda)$ we have
\begin{equation*}
\sup_{[Re>\sigma]} \left| \sum_{n=1}^{N} a_{n} e^{-\lambda_{n}s } \right| \le C_{1} \|D\|_{\infty}.
\end{equation*}
\end{corollary}

\begin{proof}  Let us write $S_{N}(it):=\sum_{n=1}^{N} a_{n} e^{-it \lambda_{n}}$ for simplicity and fix $\sigma>0$. Then by Abel summation and Theorem \ref{BohrstheoremLandau22}, choosing $\delta=\sigma$ in $(LC)$, we for all $t \in \mathbf{R}$ and $N \in \mathbf{N}$  have
\begin{align*}
\left|\sum_{n=1}^{N} a_{n} e^{-\lambda_{n}(2\sigma+it)} \right| & = \left| S_{N}(it)e^{-\lambda_{N}2\sigma}+ \sum_{n=1}^{N-1} S_{n}(it) \left(e^{-\lambda_{n}2\sigma}-e^{-\lambda_{n+1}2\sigma}\right) \right| \\ &\le C\|D\|_{\infty}\left(e^{\sigma \lambda_{N}} e^{-2\lambda_{N}\sigma}+ \frac{1}{2\sigma}\sum_{n=1}^{N-1} e^{\sigma \lambda_{n}} \int_{\lambda_{n}}^{\lambda_{n+1}}e^{-2\sigma x} dx \right) \\ &\le C\|D\|_{\infty}\left(e^{-\lambda_{N}\sigma}+ \frac{1}{2\sigma}\sum_{n=1}^{N-1}  \int_{\lambda_{n}}^{\lambda_{n+1}}e^{-\sigma x} dx \right)\\ &\le C\|D\|_{\infty}\left(1+ \frac{1}{2\sigma}\int_{0}^{\infty}  e^{-\sigma x} dx\right)=C\|D\|_{\infty}\left(1+\frac{1}{2\sigma^{2}}\right). \qedhere
\end{align*}
\end{proof}
\subsection{Bohr's condition}
We already know from Theorem \ref{BohrstheoremLandau22} that if $(BC)$ holds for $\lambda$, then $\lambda$ satisfies Bohr's theorem, since $(BC)$ implies $(LC)$ (Remark \ref{connections}). But the stronger assumption $(BC)$ improves the bound for the norm of $S_{N}$.
\begin{theorem}\label{Bohrstheorem} Let $(BC)$ hold for $\lambda$. Then there is a constant $C=C(\lambda)>0$ such that for all $D=\sum a_{n}e^{-\lambda_{n}s} \in \mathcal{D}_{\infty}^{ext}(\lambda)$ and all $N \in \mathbf{N}$ with $N\ge 2$
\begin{equation*}
\sup_{[Re>0]} \left| \sum_{n=1}^{N} a_{n} e^{-\lambda_{n}s } \right| \le C \lambda_{N} \|D\|_{\infty}.
\end{equation*}
\end{theorem}
\begin{proof}[Proof of Theorem \ref{Bohrstheorem}]Again w.l.o.g. we assume that $\lambda_{n+1}-\lambda_{n}\le 1$ following the same procedure as in the proof of Theorem \ref{BohrstheoremLandau22}. Then choosing  $k_{N}=\frac{1}{\lambda_{N}}$, $N\ge 2$ (since $\lambda_{1}=0$ is possible), by Theorem \ref{keylemma} we obtain
\begin{equation*}
\|S_{N}\|\le \frac{C_{1}}{k_{N}} \left(\frac{\lambda_{N+1}}{\lambda_{N+1}-\lambda_{N}} \right)^{k_{N}}\le  C_{2} \lambda_{N+1}^{k_{N}} \lambda_{N} \le C_{3} \lambda_{N}.\qedhere
\end{equation*}

\end{proof}
\begin{remark} For particular cases the bounds in Theorem \ref{Bohrstheorem} (and Theorem \ref{BohrstheoremLandau22}) may be bad (which is not surprising since this is an abstract result for all $\lambda$'s satisfying $(BC)$ respectively $(LC)$). For instance in the case $\lambda=(n)=(0,1,2,\ldots)$, since the projection $S_{N}(f)=\sum_{n=0}^{N} c_{n}(f)z^{n} \colon H_{\infty}(\mathbf{D})\to H_{\infty}(\mathbf{D})$ corresponds to convolution with the Dirichlet kernel $D_{N}(z)=\sum_{n=-N}^{N} z^{n}$ (after using the identification $H_{\infty}(\mathbf{D})=H_{\infty}(\mathbf{T})$), we obtain that $\|S_{N}\|=\|D_{N}\|_{1}\sim \log(N)$. To our knowledge the question about optimality of the bounds in the ordinary case $\lambda=(\log n)$ is still open.
\end{remark}
To put it differently the conditions $(BC)$ and $(LC)$ states that the sequence $\left(\log\left(\frac{1}{\lambda_{n+1}-\lambda_{n}}\right)\right)$ increases at most linearly respectively exponentially, and the quality of the growth gives different bounds for $\|S_{N}\|$. We consider now $\lambda's$ whose growth is somewhat in between:
\begin{equation} \label{polynomialgrowth}
\exists~ l,~d>0~ \forall ~\delta>0 ~\exists ~C>0~ \forall ~n \in \mathbf{N} \colon ~~\lambda_{n+1}-\lambda_{n}\ge C e^{-(l+\delta)\lambda_{n}^{d}}.
\end{equation}
Clearly (\ref{polynomialgrowth}) implies $(LC)$ and so Theorem \ref{BohrstheoremLandau22} holds, but for this class of frequencies Theorem \ref{keylemma} gives an improved bound for $\|S_{N}\|$. Recall that $\lambda=(\sqrt{\log(n)})$ satisfies (\ref{polynomialgrowth}) with $d=2$ (see proof of Remark \ref{connections}).
\begin{theorem} If $\lambda$ satisfies $(\ref{polynomialgrowth})$ with $d>0$, then there is a constant $C=C(d,\lambda)$ such that for all $D(s)=\sum a_{n}e^{-\lambda_{n}s} \in \mathcal{D}_{\infty}^{ext}(\lambda)$ and $N \in \mathbf{N}$ with $N\ge 2$
\begin{equation*}
\sup_{[Re>0]} \left| \sum_{n=1}^{N} a_{n} e^{-\lambda_{n}s } \right| \le C \lambda^{d}_{N} \|D\|_{\infty}.
\end{equation*}
\end{theorem}
\begin{proof}[Proof of Theorem \ref{Bohrstheorem}] As before w.l.o.g. we assume that $\lambda_{N+1}-\lambda_{N}\le 1$. Then choosing  $k_{N}=\frac{1}{\lambda_{N}^{d}}$, $N\ge 2$, we obtain (with Theorem \ref{keylemma}) 
\begin{equation*}
\|S_{N}\| \le C_{1} \lambda_{N+1}^{k_{N}} \lambda_{N}^{d} \le C_{2} \lambda_{N}^{d}. \qedhere
\end{equation*}
\end{proof}
\subsection{$\mathbf{Q}$-linearly independent frequencies}
In \cite{Bohr3} Bohr proves that $\mathbf{Q}$-linearly independent $\lambda$'s satisfy the equality
\begin{equation*}
\sigma_{b}^{ext}=\sigma_{a}
\end{equation*}
for all somewhere convergent $\lambda$-Dirichlet series. In this section we give an alternative proof to Bohr's using Proposition \ref{approx} and the so-called Kronecker's theorem, which states that the set $\left\{ (e^{-\lambda_{n}it}) \mid t \in\mathbf{R}\right\}$ is dense in $\mathbf{T}^{\infty}$, whenever the real sequence $(\lambda_{n})$ is $\mathbf{Q}$-linearly independent. The latter is equivalent to the fact that for every choice of complex coefficient $a_{1}, \ldots, a_{N}$ the equality
\begin{equation} \label{kronecker}
\sup_{[Re>0]}\left| \sum_{n=1}^{N} a_{n} e^{-\lambda_{n}s} \right|=\sum_{n=1}^{N} |a_{n}|.
\end{equation}
holds. For a proof of the equivalence of Kronecker's theorem and $(\ref{kronecker})$ see \cite[\S VI.9]{Katznelson} and for a proof of Kronecker's theorem see e.g. \cite[\S 3.1, Example 3.7]{DefantSchoolmann} or again \cite[\S VI.9]{Katznelson}.
\begin{theorem} \label{independent} Let $D=\sum a_{n} e^{-\lambda_{n}s} \in \mathcal{D}^{ext}_{\infty}(\lambda)$ and let $\lambda$ be $\mathbf{Q}$-linearly independent. Then $(a_{n})\in \ell_{1}$ and $\|(a_{n})\|_{1}= \|D\|_{\infty}$. Moreover, isometrically
\begin{equation*}
\mathcal{D}^{ext}_{\infty}(\lambda)=\mathcal{D}_{\infty}(\lambda)=\ell_{1}
\end{equation*}
and $\lambda$ satisfies Bohr's theorem. In particular,
\begin{equation*}
\sup_{N\in \mathbf{N}} \sup_{[Re>0]} \left| \sum_{n=1}^{N} a_{n} e^{-\lambda_{n}s} \right| =\|D\|_{\infty}.
\end{equation*}
\end{theorem}
\begin{proof} Let $f$ be the extension of $D$. Then by Proposition \ref{approx} for every $\sigma>0$ the polynomials
\begin{equation*}
R^{1}_{x}(D)=\sum_{\lambda_{n}<x} a_{n}
\left(1-\frac{\lambda_{n}}{x}\right)e^{-\lambda_{n}s} 
\end{equation*}
converge to $f$ uniformly on $[Re=\sigma]$ as $x\to \infty$. Together with $(\ref{kronecker})$ we for all $N \in \mathbf{N}$ obtain
\begin{align*}
\sum_{n=1}^{N} |a_{n}|&=\sup_{ \sigma >0 }\lim_{x\to \infty} \sum_{n=1}^{N}|a_{n}|\left(1-\frac{\lambda_{n}}{x}\right) e^{-\lambda_{n}\sigma}
\\ & \le \sup_{ \sigma > 0}\lim_{x\to \infty} \sum_{\lambda_{n} < x} |a_{n}|\left(1-\frac{\lambda_{n}}{x}\right) e^{-\lambda_{n}\sigma} \\ &=\sup_{ \sigma>0 }\lim_{x\to \infty} \sup_{t \in \mathbf{R}} \left| \sum_{\lambda_{n} < x} a_{n}\left(1-\frac{\lambda_{n}}{x}\right) e^{-\lambda_{n}(\sigma+it)} \right| \\ &=\sup_{\sigma>0} \sup_{t \in \mathbf{R}} \left|f(\sigma+it)\right|=\|D\|_{\infty}.
\end{align*}
So $(a_{n}) \in \ell_{1}$ with $\|(a_{n})\|_{1}\le \|D\|_{\infty}$ and $\sigma_{a}(D)\le 0$.  Hence
\begin{equation*}
\|D\|_{\infty}=\sup_{\sigma>0}\sup_{[Re>\sigma]} |D(s)|\le \sup_{\sigma>0} \sum_{n=1}^{\infty} |a_{n}|e^{-\sigma\lambda_{n}}\le \|(a_{n})\|_{1}. \qedhere
\end{equation*}
\end{proof}
 Let us summarize the results of Theorem \ref{BohrstheoremLandau22}, Theorem \ref{independent} and $(\ref{uiui})$.
\begin{remark} A frequency $\lambda$ satisfies Bohr's theorem if one the following conditions holds:
\begin{enumerate}
\item[(1)] $L(\lambda)=0$,
\item[(2)] $\lambda$ is $\mathbf{Q}$-linearly independent,
\item[(3)] $(LC)$.
\end{enumerate}
\end{remark}
In the theory of ordinary Dirichlet series, given $D=\sum a_{n}n^{-s}$, the so called $m$-homogeneous part of $D$ is the (formal) sum $\sum a_{n} n^{-s}$, where $a_{n}\ne 0$ implies $n$ only has $m$ prime factors counting multiplicity. Recall that $\mathcal{D}_{\infty}((\log n))$ and the space $H_{\infty}(B_{c_{0}})$ of all holomorphic and bounded functions on the open unit ball of $c_{0}$ are isometrically isomorphic via Bohr's transform
\begin{equation*}
\mathcal{B} \colon H_{\infty}(B_{c_{0}}) \to \mathcal{D}_{\infty}((\log n)), ~~f \mapsto \sum a_{n} n^{-s},
\end{equation*}
where $a_{n}:=c_{\alpha}(f)$ (the $\alpha$th Taylor coefficient of $f$) whenever $n=p^{\alpha}$ in its prime number decomposition. This identification links the space of all $m$-homogeneous Dirichlet series $\mathcal{D}_{\infty}^{(m)}((\log n))$ to the space of $m$-homogenous polynomials (or equivalently bounded $m$-linear forms) on $c_{0}$ (see \cite[\S 2, \S 3]{Defant}). In particular,  $\mathcal{D}^{(1)}_{\infty}((\log n))$ equals the space of all $1$-linear forms on $c_{0}$. Hence $\mathcal{D}_{\infty}^{(1)}((\log n))=c_{0}^{\prime}=\ell_{1}$.  Since $(\log p_{n})$, where is $p_{n}$ is the $n$th prime number, is $\mathbf{Q}$-linearly independent, Theorem \ref{independent} recovers this result.
\begin{corollary}
\begin{equation*}
\mathcal{D}_{\infty}^{(1)}((\log n))=\mathcal{D}_{\infty}((\log p_{n}))=\ell_{1}.
\end{equation*}
\end{corollary}
\subsection{A Montel theorem}In \cite[Lemma 18]{Bayart} Bayart proves that every bounded sequence $(D^{N})\subset \mathcal{D}_{\infty}(\log n)$ allows a subsequence $(D^{N_{k}})$ and some $D \in \mathcal{D}_{\infty}((\log n))$ such that $(D^{N_{k}})$ converges uniformly to $D$ on $[Re> \varepsilon]$ for all $\varepsilon>0$; a fact which is called 'Montel theorem' and extends to the following classes of $\lambda$'s.

\begin{theorem} \label{montel} Let $\lambda$ be a frequency statisfying $(LC)$, $L(\lambda)=0$ or let $\lambda$ be $\mathbf{Q}$-linearly independent. Let $(D_{j})\subset \mathcal{D}_{\infty}(\lambda)$ be a bounded sequence with Dirichlet coefficients $(a_{n}^{j})$. Then there is a subsequence $(j_{k})$ such that $(D_{j_{k}})$ converges uniformly on $[Re>\varepsilon]$ to $D=\sum a_{n} e^{-\lambda_{n}s} \in \mathcal{D}_{\infty}(\lambda)$ for all $\varepsilon>0$, where $a_{n}:=\lim_{k\to \infty} a_{n}^{j_{k}}$.
\end{theorem}
\begin{proof}
By Corollary \ref{almost periodic} we for all $n$, $j \in \mathbf{N}$ have
\begin{equation*}
|a_{n}^{j}|\le \|D_{j}\|_{\infty}\le \sup_{j } \|D_j\|_{\infty}=:C_{1}<\infty. \end{equation*}
Hence by diagonal process we find a subsequence $(j_{k})$ such that $\lim_{k} a_{n}^{j_{k}}=:a_{n}$ exists for all $n \in \mathbf{N}$. We (formally) define $D:=\sum a_{n} e^{-\lambda_{n}s}$. If $L(\lambda)=0$, then $\sigma_{a}(D)\le 0$ and $\sigma_{a}(D_{j})\le 0$ for all $j$, since the Dirichlet coefficients are bounded, and the claim follows easily. In the remaining cases Theorem \ref{BohrstheoremLandau22} respectively Theorem \ref{independent} together with Proposition \ref{BohrCahen} applied to the sequence $(D_{j_{k}})$ gives the claim. Indeed,  let $\varepsilon>0$ and let us first assume that $\lambda$ fulfils $(LC)$. Then for all $k$ and $N$ by Theorem \ref{BohrstheoremLandau22}
\begin{equation*}
\sup_{[Re>0]} \left| \sum_{n=1}^{N} a_{n}^{j_{k}} e^{-\lambda_{n}s}\right|\le C(\varepsilon)e ^{\varepsilon\lambda_{N}} \|D_{j_{k}}\|_{\infty}\le C(\varepsilon) C_{1}e^{\varepsilon\lambda_{N}}. 
\end{equation*}
Now Proposition \ref{BohrCahen} implies that $(D_{j_{k}})$ converges to $D$ on $[Re>2\varepsilon]$. If $\lambda$ is $\mathbf{Q}$-linearly independent, the claim follows in the same way replacing Theorem \ref{BohrstheoremLandau22} by Theorem \ref{independent}.
\end{proof}

\section{About completeness } \label{completenesschapter}
Recall that from Corollary \ref{almost periodic} we know that $(\mathcal{D}^{ext}_{\infty}(\lambda), \|\cdot\|_{\infty})$ is a normed space. In this section we face completeness. We first state sufficient conditions on $\lambda$ for completeness of $\mathcal{D}_{\infty}(\lambda)$ and $\mathcal{D}^{ext}_{\infty}(\lambda)$. Then we give a construction of $\lambda's$ for which $\mathcal{D}_{\infty}(\lambda)$ fails to be complete. We like to mention that in \cite{Maestre} it is already proven that $(BC)$ is sufficient for completeness of $\mathcal{D}_{\infty}(\lambda)$ by introducing the following condition, which is equivalent to $(BC)$:\begin{equation} \label{conditionMaestre}
\exists~ p\ge 1: \inf_{n \in \mathbf{N}} e^{p\lambda_{n+1}}-e^{p\lambda_{n}} >0.
\end{equation}
Their proof (see \cite[\S 2, Proposition 2.1]{Maestre}) shows that given $(\ref{conditionMaestre})$ the choice $l(\lambda):=p$ succeeds in $(BC)$ and given $l(\lambda)$ from $(BC)$ the choice $p:=l+\delta$ for every $\delta>0$ is admissable for $(\ref{conditionMaestre})$.
\medskip

Recall that being an isometric subspace of $H_{\infty}[Re>0]$ the spaces  $\mathcal{D}^{ext}_{\infty}(\lambda)$ and $\mathcal{D}_{\infty}(\lambda)$ are complete if and only if they are closed in $H_{\infty}[Re>0]$ (the space of all bounded and holomorphic functions on $[Re>0]$). 
\begin{theorem} \label{completeness}
If $L(\lambda)<\infty$, then $\mathcal{D}^{ext}_{\infty}(\lambda)$ is complete. The space $\mathcal{D}_{\infty}(\lambda)$ is a Banach space if one of the following conditions hold:
\begin{enumerate}
\item[1)] $\lambda$ is $\mathbf{Q}$-linearly independent,
\item[2)] $L(\lambda)=0$,
\item[3)] $\sigma_{b}^{ext}=\sigma_{c}$ and $L(\lambda)<\infty$. In particular, this holds for $\lambda$'s satisfying
\begin{enumerate}
\item[3.1)] $(BC)$,
\item[3.2)] $(LC)$ and $L(\lambda)<\infty$.
\end{enumerate}
\end{enumerate}
Moreover, all of the stated conditions are not necessary for completeness.
\end{theorem}
Because of the different nature of the stated sufficient conditions on $\lambda$, it seems like we are far away from a characterization. In particular, it would be interesting to find a condition on $\lambda$ sufficient for $\sigma_{b}^{ext}=\sigma_{c}$, which is weaker than $(LC)$.
\begin{proof}[Proof of Theorem \ref{completeness}]
If $\lambda$ is $\mathbf{Q}$-linearly independent, then $\mathcal{D}_{\infty}(\lambda)=\ell_{1}$ as Banach spaces by Theorem \ref{independent}. So let us assume $L(\lambda)<\infty$. Then we claim that $\mathcal{D}^{ext}_{\infty}(\lambda)$ is a closed subspace of  $H_\infty([Re > 0])$. Indeed if $(D^{K})$ is a sequence in $\mathcal{D}^{ext}_{\infty}(\lambda)$ with Dirichlet coefficients $(a_{n}^{K})$, which converges
 to some $f \in H_{\infty}[Re>0]$, then $|a_{n}^{K}| \le \|D^{K}\|_{\infty}$  for all $n$, $K \in \mathbf{N}$ (Corollary \ref{almost periodic}). Hence the limits $a_{n}:=\lim_{K} a_{n}^{K}$ exist and $(a_{n})$ is bounded.  So the Dirichlet series $D:=\sum a_{n}e^{-\lambda_{n}s}$ converges absolutely on $[Re>L(\lambda)]$ and we claim that $D$ and $f$ coincide on $[Re>L(\lambda)]$. Let $s\in [Re>L(\lambda)]$, $Re ~s= \sigma$  and  $\varepsilon>0$. Then there is $K_{0}$ such that $|a_{n}-a_{n}^{K}|\le \varepsilon$ for all $K\ge K_{0}$ and all $n\in \mathbf{N}$. Then for such $K$ we obtain for large  $N$
\begin{align*}
|D(s)-f(s)|\le &\left| D(s)-\sum_{n=1}^{N}a_{n}e^{-\lambda_{n}s}\right| +\left|\sum_{n=1}^{N}(a_{n}-a_{n}^{K})e^{-\lambda_{n}s} \right|\\ & +\left|\sum_{n=1}^{N}a_{n}^{K} e^{-\lambda_{n}s} -D^{K}(s)\right| +\left|D^{K}(s)-f(s) \right| \\ &\le \varepsilon+\varepsilon \sum_{n=1}^{\infty} e^{-\lambda_{n}\sigma} +\varepsilon+\varepsilon,
\end{align*}
which implies $D(s)=f(s)$ on $[Re>L(\lambda)]$. Hence  $D\in \mathcal{D}^{ext}_{\infty}(\lambda)$ with extension $f$ and $\mathcal{D}_{\infty}^{ext}(\lambda)$ is complete, which coincides with $\mathcal{D}_{\infty}(\lambda)$ assuming $\sigma_{b}^{ext}=\sigma_{c}$.\\
Now we come to the 'Moreover' part. Since $L((\log n))=1$ and $\mathcal{D}_{\infty}((\log n))$ is complete, the condition '$L(\lambda)=0$' is not necessary. Moreover, the frequency defined by $\lambda_{2n}=n+e^{-n^{2}}$ and $\lambda_{2n-1}=n$ does not satisfy $(BC)$ but $L(\lambda)=0$. Hence the condition $(BC)$ is not necessary. The frequency defined by $\lambda_{2n}=n+e^{-e^{n^{2}}}$ and $\lambda_{2n-1}=n$ does not satisfy $(LC)$ but $L(\lambda)=0$. Finally, by choosing a $\mathbf{Q}$-linearly independent $\lambda$ increasing slowly enough we see that the condition '$L(\lambda)<\infty$' is not necessary.
\end{proof}
On the other hand in the following sense there are infinitely $\lambda$'s for which $\mathcal{D}_{\infty}(\lambda)$ fails to be complete.
\begin{theorem} \label{incomplete}
Let $\lambda$ be a frequency. Then there is a strictly increasing sequence $(s_{n})$ of natural numbers such that the space $\mathcal{D}_{\infty}(\eta)$, where the frequency $\eta$ is obtained by ordering the set
\begin{equation} \label{Neder}
\left\{ \lambda_{n}+\frac{j}{s_{n}}(\lambda_{n+1}-\lambda_{n}) \mid n \in \mathbf{N}, ~j=0, \ldots s_{n}-1\right\}
\end{equation}
increasingly, is not complete.
\end{theorem}
\begin{proof} First we explain why it is sufficient to assume $\lambda_{n+1}-\lambda_{n}\le 1$ for all $n$. Therefore, suppose $\sup_{n \in \mathbf{N}} \lambda_{n+1}-\lambda_{n}>1$. Then for each $n$ we consider the interval $[\lambda_{n},\lambda_{n+1}]$ and add equidistantly new numbers to $\lambda$ in $[\lambda_{n},\lambda_{n+1}]$, such that the distance of these new numbers (inclusive the edge point $\lambda_{n}$ and $\lambda_{n+1}$) is less than $1$. Since for each interval  $I_{n}$ we only add finitely many new numbers we obtain in this way a new frequency (with subsequence $\lambda$), say $\widetilde{\lambda}$,  satisfying $\widetilde{\lambda_{n+1}}-\widetilde{\lambda_{n}}\le 1$ for all $n$. Suppose now that we are able to find a sequence $(\widetilde{s_{n}})$ for $\widetilde{\lambda}$ as stated in the theorem and we denote by $\widetilde{\eta}$ the corresponding new frequency. Then we already know that $\mathcal{D}_{\infty}(\widetilde{\eta})$ is not complete. Now we want to add again more numbers to $\widetilde{\eta}$ in such a way that the frequency $\eta$ obtained is of the form $(\ref{Neder})$ for some suitable sequence $(s_{n})$. Then clearly $\mathcal{D}_{\infty}(\eta)$ remains incomplete.
Note that each interval $[\lambda_{n}, \lambda_{n+1}]$ contains now finitely many intervals of the form $[\widetilde{\lambda_{j}},\widetilde{\lambda_{j+1}}]$. By assumption each of these interval is decomposed into  equidistant parts by $\widetilde{s_{j}}$. Since there are only finitely many of them, choose the smallest distance and add (finitely) many numbers to the interval $[\lambda_{n}, \lambda_{n+1}]$ such that within $[\lambda_{n}, \lambda_{n+1}]$ all added numbers are equidistant (this is always possible). Denoting by $s_{n}$ the number in $[\lambda_{n},\lambda_{n+1}]$ we add to $\lambda$  by this procedure, we obtain the desired frequency $\eta$ of the form (\ref{Neder}).\\
So we may assume that $\sup_{n} \lambda_{n+1}-\lambda_{n}\le 1$. Now let $I_{k}:=\{\lambda_{n} \mid k\le \lambda_{n}<k+1\}$, $k\in \mathbf{N}_{0}$, and $x>0$ arbitrary. If $|I_{k}|\ne 0$, where $|I_{k}|$ denotes the number of elements in $I_{k}$, we define $b_{k}=e^{-xk}|I_{k}|^{-1}$, and if $|I_{k}|=0$, then $b_{k}=0$. Note that $|I_{k}|<\infty$, since $\lambda$ is strictly increasing. If $\lambda_{n} \in I_{k}$ for some $k$ we define $r_{n,k}$ to be the largest natural number smaller than $e^{e^{2x\lambda_{n}}|I_{k}|}$. Following the spirit of Neder in \cite[\S 1]{Neder} we define the new frequency $\eta$  given by the following set
\begin{equation*}
\bigcup_{k\in \mathbf{N}} \left\{ \lambda_{n}+\frac{j}{2r_{n,k}}(\lambda_{n+1}-\lambda_{n}) \mid \lambda_{n}\in I_{k}, ~j=0, \ldots 2r_{n,k}-1\right\}.
\end{equation*}
 Then we (formally) define  the $\eta$-Dirichlet series $D=\sum a_{m}e^{-\eta_{m} s},$ by setting
$a_{m}=b_{k}\frac{1}{r_{n,k}-j}$, if $\eta_{m}=\lambda_{n}+\frac{j}{2r_{n,k}}(\lambda_{n+1}-\lambda_{n})$ for some $\lambda_{n}\in I_{k}$ and $j\in \{1,\ldots, 2r_{n,k}-1\}\setminus \{r_{n}\}$. The remaining $a_{m}$'s equal zero. Now we claim that $\sigma_{c}(D)\ge x>0$, which shows $D \notin \mathcal{D}_{\infty}(\eta)$. If $\lambda_{n}\in I_{k}$, then by the choice of $b_{k}$ and $r_{n,k}$
\begin{align*}
\sum_{\lambda_{n}\le \eta_{m} <\lambda_{n}+\frac{1}{2}(\lambda_{n+1}-\lambda_{n})} a_{m}e^{-x \eta_{m}}&=b_{k} \sum_{j=1}^{r_{n,k}-1} \frac{1}{r_{n,k}-j} e^{-x (\lambda_{n}+\frac{j}{2r_{n,k}}(\lambda_{n+1}-\lambda_{n}))} \\ &\ge b_{k}e^{-x} e^{-x \lambda_{n}}  \sum_{j=1}^{r_{n,k}-1} \frac{1}{j}\ge b_{k} e^{-x} e^{-x \lambda_{n}}  \log(r_{n,k}) \\ & \ge b_{k}e^{-x} e^{-x \lambda_{n}} \frac{1}{4}e^{\lambda_{n}2x}|I_{k}|= \frac{e^{-x}}{4}e^{-xk}e^{\lambda_{n}x}\ge \frac{e^{-x}}{4}>0.
\end{align*}
Hence $D$ does not converge in $x$. Now we claim that $\mathcal{D}_{\infty}(\eta)$ is not complete. Therefore recall (see \cite[\S 4]{EHP}) that the Fej\'{e}r polynomials $F_{m}(z):= \sum_{j=1}^{2m-1} \frac{1}{m-j} z^{j}$, $m \in\mathbf{N}$, where $\frac{1}{0}:=0$,
satisfy 
\begin{equation*}
\sup_{m \in \mathbf{N}} \sup_{|z|<1} |F_{m}(z)| =:C<\infty.
\end{equation*}

Now consider the sums
\begin{align*}
D^{K}(s)&:=\sum_{k=1}^{K} b_{k} \sum_{\lambda_{n}\in I_{k}} e^{-\lambda_{n}s}F_{r_{n,k}}(e^{-s\frac{1}{2r_{n}}(\lambda_{n+1}-\lambda_{n})})\\ &=\sum_{n=1}^{K}b_{k} \sum_{j=1}^{2r_{n,k}-1} \frac{1}{r_{n,k}-j}e^{-\left(\lambda_{n}+\frac{j}{2r_{n,k}}(\lambda_{n+1}-\lambda_{n})\right)s},
\end{align*}
which are partial sums of $D$ defined before and $(D^{K}) \subset \mathcal{D}_{\infty}(\eta)$. Moreover, for each $s \in  [Re>0]$ and $K>L$
\begin{align*}
|D^{K}(s)-D^{L}(s)|\le C \sum_{k=K}^{L} b_{k}  |I_{k}|\le \sum_{k=K}^{\infty} (e^{-x})^{k}
\end{align*}
 and so $(D^{K})$ is Cauchy in $\mathcal{D}_{\infty}(\eta)$. Assuming now that  $\mathcal{D}_{\infty}(\eta)$ is  complete, the limit of $(D^{K})$ in $\mathcal{D}_{\infty}(\eta)$ is $D$,  since $a_{m}=\lim_{K} a_{m}(D^{K})$. But we already know that $D \notin \mathcal{D}_{\infty}(\eta)$, which gives a contradiction. 
\end{proof}

\textbf{Acknowledgement:}
 I thank Antonio P\'{e}rez for several discussions on Dirichlet series, which in particular led to Theorem \ref{independent}, and I thank the referees for careful reading and their helpful critique.

\vspace{\baselineskip}

\end{document}